\documentclass[english]{smfart}
\usepackage[T1]{fontenc}
\usepackage[utf8x]{inputenc}
\usepackage{graphicx}
\usepackage{amssymb}
\usepackage{amsmath}
\usepackage{amsthm}
\usepackage{color}
\usepackage{relsize}
\usepackage[all]{xy}
\usepackage{amsfonts}
\usepackage{mathrsfs}
\usepackage{latexsym}
\usepackage{geometry} 
\usepackage{textcomp}
\usepackage[disable]{todonotes}
  \usepackage[english]{babel}
\usepackage{tikz}
\usepackage{tikz-cd}
\usetikzlibrary{matrix,arrows}
\usepackage{cleveref}

\newcommand\ZZ{\mathbb{Z}} 
\newcommand\QQ{\mathbb{Q}} 
\newcommand\RR{\mathbb{R}} 
\newcommand\CC{\mathbb{C}} 

\newcommand\FP{\mathbb{F}_p}

\newcommand\fleche{\longrightarrow} 

\DeclareMathOperator{\Ker}{Ker}

\DeclareMathOperator{\tr}{tr}

\DeclareMathOperator{\Ht}{ht}
\DeclareMathOperator{\Hom}{Hom}

\DeclareMathOperator{\Spec}{Spec}
\DeclareMathOperator{\Spf}{Spf}
\DeclareMathOperator{\Vect}{Vect}
\DeclareMathOperator{\id}{id}
\DeclareMathOperator{\Gal}{Gal}

\DeclareMathOperator{\Lie}{Lie}
\DeclareMathOperator{\Hdg}{Hdg}

\DeclareMathOperator{\Newt}{\mathcal N\!ewt}

\DeclareMathOperator{\End}{End}

\DeclareMathOperator{\PR}{PR}
\DeclareMathOperator{\GL}{GL}
\DeclareMathOperator{\Char}{Char}
\DeclareMathOperator{\Diff}{Diff}

\theoremstyle{definition} 
\newtheorem{defin}{Definition}[section]
\newtheorem{hypothese}[defin]{Hypothesis}

\theoremstyle{plain} 
\newtheorem{theor}[defin]{Theorem}
\newtheorem{lemm}[defin]{Lemma}  
\newtheorem{prop}[defin]{Proposition}

\newtheorem{cor}[defin]{Corollary}

\theoremstyle{remark} 
\newtheorem{rema}[defin]{Remark}
\newtheorem{example}[defin]{Example}

\begin{document}

\title{On the geometry of the Pappas-Rapoport models for PEL Shimura varieties}
\author{St\'ephane Bijakowski}

\address{Centre de mathématiques Laurent Schwartz, École Polytechnique, 91128 Palaiseau, FRANCE
}

\email{stephane.bijakowski@polytechnique.edu}

\author{Valentin Hernandez}

\address{Laboratoire Mathématiques d'Orsay, Bâtiment 307,
Université Paris-Sud, 91405 Orsay, FRANCE
}

\email{valentin.hernandez@math.crns.fr}\date{}
\maketitle

\section{Introduction}
Shimura varieties have been at the heart of arithmetic since their introduction by Goro Shimura, later generalized by Pierre Deligne \cite{DeShimura}. Nowadays they are a powerful 
geometric tool for the Langlands program. As algebraic varieties over a number field $E$, their \'etale cohomology in endowed both with an action of $G_E = \Gal(\overline E/E)$ and of 
the adelic points of the underlying reductive group $G$ : understanding the relations between the two actions is the way to realize \textit{geometrically} (cases of) the association of Galois 
representations to automorphic representations. This strategy was first realized by Eichler-Shimura and Deligne for the modular curves, and was later generalised in broader directions, for 
higher dimensional Shimura varieties (\cite{LRZ}, \cite{KotJams}, \cite{HT}, \cite{Shin},\cite{ParisBook}...) where the previous arithmetic and analytic relations has revealed very 
complex issues.

One of the ideas to realize this correspondance is the Langlands-Kottwitz method, for which we need to relate the number of points (modulo $p$) of our Shimura variety (itself related to the \'etale cohomology of the variety), to some orbital integrals of $G$, itself related in a somewhat indirect way, but now classical, to automorphic representations. Thus, to make sense of the number of points, we need to find a way to reduce the given Shimura variety modulo $p$, i.e. we need to find a \textit{good} integral model of it. When the Shimura Variety is of P.E.L. type, meaning more or less that it is a moduli space of abelian varieties with some extra structure, the simplest idea is to extend this modular description from $E$ to $\mathcal O_E$, or at least to $\mathcal O_{E_\mathfrak p}$, a $p$-adic completion of $\mathcal O_E$. In the first case of the modular curve, this has been extensively studied, for example in \cite{DR} or \cite{KM}, in which very satisfying integral models are introduced, for all the interesting levels at $p$, $\Gamma_0(p^n),\Gamma_1(p^n),\Gamma(p^n)...$. A remark regarding the definition of integral models is that the level away from $p$ is easy to deal with. Also, by \textit{satisfying} integral models here we mean with as little singularities as possible. For example when the level at $p$ is maximal, the integral model of the modular curves is smooth, and in general level they are regular. Kottwitz then deeply generalised this in the case of P.E.L Shimura varieties, provided that the Shimura datum was \textit{unramified} at $p$, meaning that both the group is unramified at $p$ (and has a suitable integral model), and the level is hyperspecial at $p$. 

The problem of defining good integral models both with deeper level at $p$, or for ramified Shimura datum, has since been deeply studied. For a selection we mention the work of 
Harris-Taylor \cite{HT} which study specific Shimura varieties for which the method of Katz-Mazur still applies for deeper level, work of Pappas-Rapoport \cite{P-R} for cases where the 
Shimura datum is ramified, and almost any paper of Lan (for example \cite{Lan,LanSigma,LanSplitting}) for generalisation in both directions.
In this article we study a specific class of P.E.L. integral models, with "maximal" level at $p$ (in a specific sense), and for which the group is ramified at $p$. We take the definition given in \cite{P-R}, also referred to as \textit{splitting models} in the literature, and study the local and global geometry of these models. Our results depend on the ramification of $p$ on the Shimura datum. Precisely, as explained in section \ref{sectdecpol}, there is a finite set $\mathcal P$ of "primes" $\pi$ above $p$. These primes fall in one of the following categories : (C),(AL),(AU),(AR), where the first one is the category of primes of symplectic type, and the last three are of unitary type.\footnote{We exclude all type D factors in our P.E.L. Shimura varieties} The last category (AR) roughly corresponds to a unitary group over a CM extension $F/F^+$, and a prime $\pi$ above $p$ in $F^+$ such that $\pi$ ramifies in $F$. Denote 
$X$ the \textit{Pappas-Rapoport model} at $p$ of a Shimura variety as in section \ref{sect2}. It lives over the ring of integers of $K$, a finite, well chosen, extension of $\QQ_p$. 
Our first result is the following (see Theorem \ref{thr229})

\begin{theor}
If no prime in $\mathcal P$ falls in case (AR), then $X$  is smooth over $\Spec(\mathcal O_{K})$.
\end{theor}

Such a result was clearly expected in \cite{P-R}, and was already proven in the case of the Hilbert modular varieties in \cite{Sasaki},\cite{RX}. Our proof is very similar, using the definitions of 
\cite{P-R} and the local study we make in section \ref{sect2}. Also it is clear that the assumption that no prime falls in case (AR) is necessary, as explained in the Appendix.

The main result of this article is a study of the special fiber $X_\kappa$ of $X$. Recalling that for a (P.E.L.) Shimura variety $S$ associated with data unramified at $p$, we can look at the 
Newton stratification of the special fiber of $S$, which we now know has all the expected properties, in particular their $\mu$-ordinary locus is open and dense (\cite{Wed1,Ham}).
In this article we study a similar question in our situation, and we investigate another natural stratification, that we call the Hodge stratification on $X_\kappa$, encoding the position of the Hodge polygon (defined in \cite{BH}). Even if we show that this stratification doesn't behave as well as expected (except in case of very small ramification – $e=2$ – and even only away from case (AR)), we prove that the open stratum, the \textit{generalised Rapoport locus}, is dense (except in case (AR), again). This locus coincides with the usual Rapoport locus in the Hilbert case, hence the denomination.

\begin{theor}
If no prime in $\mathcal P$ falls in case (AR), the generalised Rapoport locus is open and dense.
\end{theor}

We actually prove this result by hand by explicitely constructing a deformation of a $p$-divisible group to the generalised Rapoport locus, see Theorem \ref{thrRap}. Then we 
investigate the similar result for the Newton stratification. Because of our earlier results on Pappas-Rapoport data (see \cite{BH}), we know that the $\mu$-ordinary locus, which 
coïncides with the (big) open stratum of the Newton stratification, lies inside the generalised Rapoport locus. Here, we prove that it is dense, generalising work of Wedhorn (\cite{Wed1}) in the case of a ramified Shimura variety.

\begin{theor}
If no prime in $\mathcal P$ falls in case (AR), the $\mu$-ordinary locus in $X_\kappa$ is open and dense.
\end{theor}

This result actually implies the previous one, but the proof uses the density of the generalised Rapoport locus, together with the methods of deformation of $p$-divisible groups introduced in \cite{Wed1}, and relies on calculations on displays. Here we slightly simplify some arguments of \cite{Wed1}, constructing "by hand" deformations when we can. This density result extends the work of \cite{Wed1} (which deals with the unramified case). Note also the work of Wortmann (\cite{Wortmann}) for Hodge-Type Shimura varieties with good reduction at $p$ and work of He-Rapoport (\cite{HeRap}) and He-Nie (\cite{HeNie}) to compare the $\mu$-ordinary locus to EKOR strata and to reformulate this density in terms of Weyl groups. \\
Finally, we give an equivalent condition, similar to the unramified case, for the existence of an ordinary locus. Namely, we prove that the ordinary locus is non empty if and only if the prime $p$ is totally split in the reflex field, extending the result in \cite{Wed1} in the unramified case.

We would like to thank F. Andreatta and E. Goren for an interesting discussion and suggesting to have a look at \cite{Kramer}, P. Hamacher, M. Rapoport, T. Richarz for interesting discussions about this article and related works.

\section{Shimura datum, Pappas-Rapoport condition and stratifications}
\label{sect2}
\subsection{Shimura Datum} \todo{Fixer un nombre premier p  : fait avant hypothèse 2.2 $\Box$}

Let $(B,\star)$ be a finite dimensional central semisimple $\QQ$-algebra endowed with a positive involution, with center $F$, and $(V,<.,.>)$ be a non-degenerate skew hermitian $B$-module, and let $G$ be the algebraic group over $\QQ$ of (similitude)-automorphisms of $(V,<.,.>)$, i.e. representing the functor,
\[ G(R) = \{ (g,c) \in \GL(V\otimes_\QQ R) \times \mathbb G_m(R) | <gz,gz'> = c<z,z'>, \forall z \in V\otimes_\QQ R\},\]
on $\QQ$-algebras $R$.

Let $h : \CC \fleche \End_B(V_\RR)$ be a $\RR$-algebra homomorphism such that $h(\overline z) = h(z)^\star$ and the bilinear form $(\cdot,h(i)\cdot)$ on $V_\RR$ is definite positive.

\subsection{Characteristic zero moduli space}
\label{sect2.2}
Let us denote by $E$ the reflex field of the (Shimura) datum $(G,h)$; it is a number field. Fix $K \subset G(\mathbb A_f)$ a neat\footnote{This is mainly to simplify the exposition.} compact open subgroup. 

Following \cite{Lan} Definition 1.4.2.1, call $\mathcal S_K$ \todo{$S_K$ ou $X_K$ ? $S$ pour le moment, puis on montre en 2.23 que c'est la meme chose que X (PR) $\Box$}the moduli problem over $\Spec(E)$ that associate to $S$ the quasi-isogeny classes of quadruples $(A,\lambda,\iota,\eta)$, where 
$A \fleche S$ is a abelian scheme, 
$\lambda$ is a $\QQ^\times$-polarisation of $A$, $i :B \fleche \End(A)\otimes (\QQ)_S$ \todo{$O_B$ pas défini : c'est $B$ en fait $\Box$} is a morphism compatible with $\star$ and the Rosatti involution, and $\eta$ is a rational level structure of type $K$ of $A$ (see \cite{Lan} Definition 1.4.1.2 for a precise formulation). We moreover require that this quadruple satisfies the \textit{determinant condition}, see \cite{KotJams} section 5 or \cite{Lan} definition 1.3.4.1.
Then $\mathcal S_K$ is representable by a scheme over $\Spec(E)$. This is, for example, \cite{Lan} Corollary 1.4.3.7 and Corollary 7.2.3.10.

\begin{rema}
If $p$ is a good prime for $G,K$, we could give an analogous definition by $\ZZ_{(p)}$-isogeny instead of quasi-isogeny (i.e. $\QQ^\times$-isogeny), as we will do later in the text, but we would need to introduce integral data to give a meaning to good primes (see our definition in section \ref{sect24}, and \cite{Lan} sections 1.4.2, 1.4.3).
\end{rema}

From now on, we fix a prime $p$. Let us be more specific about the determinant condition when $S$ is over $\QQ_p$. First let us assume that the following hypothesis on $p$ and $B$ is satisfied.

\begin{hypothese}
\label{hyp1}
$B_{\QQ_p}$ is isomorphic to a product of matrix algebras over finite extension of $\QQ_p$, such that factors are either stable by $\star$ or exchanged two-by-two by $\star$. 
Note that we do not suppose the extensions to be unramified. This assumption excludes factors of type (D) (orthogonal factors).
\end{hypothese}

\begin{example}
If $B = F$, with $F/F^+$ a CM field with totally real field $F^+$, and $\star$ the complex conjugation, the previous hypothesis is verified as $B_{\QQ_p} = \prod_{\pi | p} F \otimes_{F^+} F^+_\pi$ where $\pi$ ranges over places over $p$ in $\mathcal O_{F^+}$ and $F^+_\pi$ is the $\pi$-adic completion of $F^+$. 
\end{example}

By hypothesis, we can decompose $B_{\QQ_p}  = \bigoplus_{i =1}^r M_{n_i}(F_i)$ where $F_i/\QQ_p$ is a finite, possibly ramified, extension. Remark that the involution $\star$ on $B$ acts on the set $\{1,\dots,r\}$, we denote $s(i)$ the image of $i$ by this involution. Denote by $E_\nu$ a $p$-adic completion of $E$, thus $E_\nu$ is a finite extension of $\QQ_p$.
If $(A,\lambda,\iota,\eta)$ is an object over $S$ in $\mathcal S_K \otimes_E {E_\nu}$, then $\omega_A = \Lie(A)^\vee$ is a $\mathcal O_S \otimes_{\QQ} B$ module, but as $S$ is over $\QQ_p$, it is a $\mathcal O_S \otimes_{\QQ_p} B_{\QQ_p}$-module, and we can thus decompose it as
\[\omega_A = \bigoplus_{i=1}^r \omega_{A,i},\]
where $\omega_{A,i}$ is a $\mathcal O_S \otimes_{\QQ_p} M_{n_i}(F_i)$-module. Using Morita equivalence, decompose $\omega_{A,i} = \mathcal O_S^{n_i}\otimes_{\mathcal O_S} \omega_i$, where\footnote{$e_i$ is the Morita projector associated to the matrix $E_{1,1}$ seen as an element of $M_{n_i}(F_i)$} $\omega_i = e_i\omega_{A,i}$ is endowed with an action of $\mathcal O_{F_i}$, and, up to an extension of scalars for $S$, we can further decompose $\omega_{i} = \bigoplus_{\tau \in \mathcal T'_i}\omega_{i,\tau}$, as locally free $\mathcal O_S$-modules, where $\mathcal T'_i = \Hom(F_i,\CC_p)$. Then the determinant condition is equivalent to asking the locally free $(\omega_{i,\tau'})$ to have fixed dimension $(d_{i,\tau'})_{i,\tau'}$, where the integers $(d_{i,\tau'})_{i,\tau'}$ are fixed by $h$ as follows. Denote $V_\CC = V_1 \oplus V_2$ the decomposition where $h(z)$ acts as $z$ (resp. $\overline z$) on $V_1$ (resp. $V_2$). Then the reflex field $E \subset \CC$ is the number field where the isomorphism class of complex $B$-representation $V_1$ is defined. It thus makes sense to consider $V_{\overline{\QQ_p}} = V_{1,\overline{\QQ_p}} \oplus  V_{2,\overline{\QQ_p}}$ as a $B_{\overline{\QQ_p}}$-representation. Using the hypothesis on $B$, decompose,
\[ V_{1,\overline{\QQ_p}} = \prod_{i=1}^r V_1^i \otimes_{F_i\otimes_{\QQ_p}\overline{\QQ_p}} (F_i\otimes_{\QQ_p} \overline{\QQ_p})^{n_i},\]
by Morita, where $V_1^i$ is a $F_i\otimes_{\QQ_p} \overline{\QQ_p}$-module that we can further decompose as,
\[ V_1^i = \prod_{\tau' \in \Hom(F_i,\overline{\QQ_p})} (V_1^i)_{\tau'}.\]
Then $d_{i,\tau'}$ is the dimension of $(V_1^i)_{\tau'}$.

\begin{rema}
\label{remadpol}
As our Shimura datum comes from an object over $\QQ$, we can check that, for all $i$, and for all $\tau',\tau''$, we have \todo{c = complex conjugation? Pas forcement : dans le cas AU/AR oui, dans le cas AL en fait $F_i = K \times K$ et $c$ échange les deux. $\Box$}
\[ d_{i,\tau''} + d_{s(i,\tau'')} = d_{i,\tau'} + d_{s(i,\tau')} = h_i,\]
is independent of $\tau'$, where $s$ is the action induced by $\star$ (in the case where to factors $i,j$ are exchanged by $\star$, recall that we set $j = s(i)$). This is for example \cite{Lan}, end of page 59.
\end{rema}

\subsection{Pappas-Rapoport data}
\label{sect23}

The goal of this section is to define a Pappas-Rapoport datum in order to define an integral model for the variety $\mathcal S_K$,\footnote{Such a datum, introduced in \cite{P-R}, is refered there as a splitting datum} which is analogous to Kottwitz determinantial condition but better behaved in ramified characteristics. We define such a datum in this section, and explain its behavior with duality.

\subsubsection{Definition}

Let $L/\QQ_p$ be a finite extension, $K$ be an extension of $\QQ_p$ containing the Galois closure of $L$, and $S$ is an $\mathcal O_K$-scheme. 
Denote $L^{ur}$ the maximal unramified subfield of $L$, and $\mathcal T = Hom(L^{ur},\CC_p)$ the set of unramified embeddings, and fix $\pi$ a uniformiser of $L$, with Eisenstein polynomial $Q$. 
In particular we can identify, sending $T$ to $\pi$,
\[ \mathcal O_{L^{ur}}[T]/(Q(T)) \simeq \mathcal O_F.\]
Let us fix an embedding $\tau$ of $L^{ur}$ into $K$, and define $\Sigma$ as the set of embeddings of $L$ into $K$ extending $\tau$. It is a set of cardiality $e$, and let us choose an ordering $\Sigma = \{\sigma_1, \dots, \sigma_{e}\}$ for this set.

Let $\mathcal{N} \to S$ be a locally free sheaf with an action $\mathcal O_F$, such that $\mathcal O_{L^{ur}}$ acts on $\mathcal{N}$ by $\tau$. We will denote by $[\pi]$ the action of $\pi$ on $\mathcal{N}$. Let $(d_{1}, \dots, d_{e})$ be a collection of integers. We recall the definition of a Pappas-Rapoport datum.

\begin{defin}
A Pappas-Rapoport datum for $\mathcal{N}$ with respect to the collection $(\sigma_i, d_i)_{i=1,\dots,e}$ consists in a filtration
$$0 = \mathcal{N}^{[0]} \subset \mathcal{N}^{[1]} \subset \dots \subset  \mathcal{N}^{[e]} = \mathcal{N}$$
such that
\begin{enumerate}
\item The $\mathcal{N}^{[j]}$ are $\mathcal O_S$-locally direct factors stable by $\mathcal O_{L}$
\item  $([\pi] - \sigma_{j} (\pi)) \cdot \mathcal{N}^{[j]} \subset \mathcal{N}^{[j-1]}$, for all $1 \leq  j \leq e$.
\item $\mathcal{N}^{[j]}/\mathcal{N}^{[j-1]}$ is locally free of rank $d_{j}$ for all $1 \leq  j \leq e$.
\end{enumerate}
\end{defin} 

\subsubsection{Duality}
\label{sect232}

Next, we want to explain the compatibility with duality for this datum. Assume that there exists a sheaf $\mathcal{E}$, locally free of rank $h$ as a $\mathcal O_S \otimes_{\mathcal O_{F^{ur}}, \tau} \mathcal O_F$-module, such that $\mathcal{N}$ is locally a direct factor of $\mathcal{E}$. Let $\mathcal{M} := (\mathcal{E} / \mathcal{N})^\vee$ ; it is a locally free sheaf over $S$, and has an action of $\mathcal O_L$ (with $\mathcal O_{L^{ur}}$ acting by $\tau$). One thus has an exact sequence

$$0 \to \mathcal{N} \to \mathcal{E} \to \mathcal{M}^\vee \to 0$$

Let us introduce some more notation. Define $\pi_i := \sigma_i(\pi)$ for $1 \leq i \leq e$, and let us introduce the polynomials for $1 \leq \ell \leq e$
\[ Q_\ell := \prod_{i=1}^{\ell} (T - \pi_i) \quad \text{and} \quad Q^\ell := \prod_{i=\ell+1}^{e} (T - \pi_i).\]
Note that the hypothesis made on $\mathcal E$ means that it is locally free as a $\mathcal O_S [T] / Q(T)$-module (with $T$ acting by $\pi$).

\begin{defin}
Let us define a complete filtration on $\mathcal{E}$ 
$$0 = \mathcal{N}^{[0]} \subset \mathcal{N}^{[1]} \subset \dots \subset  \mathcal{N}^{[e]} = \mathcal{N} \subset \mathcal{N}^{[e+1]} \subset \dots \subset \mathcal{N}^{[2e]} = \mathcal{E}$$
by the formulas
$$\mathcal{N}^{[2e-\ell]} = \left( Q^\ell (\pi) \right)^{-1} (\mathcal{N}^{[\ell]}).$$
for every $1 \leq \ell \leq e-1$. \\
A \textit{full Pappas-Rapoport datum} for $(\mathcal E,\mathcal N)$ with respect to $(\sigma_i,d_i)_{i=1,\dots,e}$ is a complete filtration of the previous form, where $(\mathcal N^{[i]})_{i=1,\dots,e}$ is a Pappas-Rapoport datum for $\mathcal N$ with respect to the same data.
\end{defin}

The conditions imposed by the Pappas-Rapoport datum imply that the inclusions $\mathcal{N}^{[e+j]} \subset \mathcal{N}^{[e+j+1]}$ are satisfied for every $0 \leq j \leq e-1$.

\begin{lemm}
Let $1 \leq j \leq e-1$ be an integer. The sheaf $\mathcal{N}^{[e+j]}$ is locally free of rank \[jh + d_1 + \dots + d_{e-j} = \dim_{\mathcal O_S}\mathcal N + h-d_e + \dots + h-d_{e-j+1}.\] Moreover, one has
$$([\pi] - \sigma_{e-j+1}(\pi)) \mathcal{N}^{[e+j]} \subset \mathcal{N}^{[e+j-1]}.$$
\end{lemm}

\begin{proof}
This is an easy computation.
\end{proof}

One deduces from this lemma that one has a Pappas-Rapoport datum for $\mathcal{M}$.

\begin{prop}
\label{propdual}
The complete filtration on $\mathcal{E}$ induces a Pappas-Rapoport datum for $\mathcal{M}$ with respect to the collection $((\sigma_1,\dots,\sigma_e),(h-d_{1}, \dots, h-d_{e}))$.
\end{prop} 

\begin{rema}
In special fiber, the situation is quite simpler. Indeed, one has simply $Q_\ell (\pi) = \pi^\ell$, $Q^\ell (\pi) = \pi^{e-\ell}$ and
$$\mathcal{N}^{[2e-\ell]} = \left( \pi^{e-\ell} \right)^{-1} (\mathcal{N}^{[\ell]})$$
for every $1 \leq \ell \leq e-1$.
\end{rema}

%
%
%

\subsubsection{Pairing}
\label{sect233}

Assume in this section that the sheaf $\mathcal{E}$ has a perfect alternating pairing $< , > : \mathcal{E} \times \mathcal{E} \to \mathcal{O}_S$. Suppose also that this pairing is compatible with the action of $\mathcal{O}_L$, i.e. that $<a \cdot x, y> = <x, a \cdot y>$ for $a \in \mathcal{O}_F$ and $x,y \in \mathcal{E}$. This forces the integer $h$ to be even; let $g$ be such that $h = 2g$. Assume moreover that $\mathcal N$ is maximally isotropic, i.e. $\mathcal N = \mathcal N^\bot$, the latter notation refering to the orthogonal of $\mathcal N$ for the considered pairing. This implies that $\mathcal N$ is locally free of rank $eg$.

\begin{prop}
Fix a Pappas-Rapoport datum for $\mathcal N$ with respect to the collection $(\sigma_i,d_i)_{i=1,\dots,e}$. There exists a complete filtration of $\mathcal E$ given by
$$0 = \mathcal{N}^{[0]} \subset \mathcal{N}^{[1]} \subset \dots \subset  \mathcal{N}^{[e]} = \mathcal{N} \subset {\mathcal{N}^{[e-1]}}^\bot \subset  \dots \subset  {\mathcal{N}^{[1]}}^\bot \subset \mathcal E$$
This filtration induces a Pappas-Rapoport datum for $\mathcal{M}$ with respect to the collection $((\sigma_1,\dots,\sigma_e),(d_{1}, \dots, d_{e}))$.
\end{prop}

\begin{proof}
Let us consider the sheaf ${\mathcal{N}^{[e-1]}}^\bot$. It is locally a direct factor of rank $2eg - (d_1 + \dots + d_{e-1})=eg+d_e$ since $d_1 + \dots + d_e = eg$. Let $x \in {\mathcal{N}^{[e-1]}}^\bot$ and $y \in \mathcal{N}$. Then $[\pi] y - \sigma_e(\pi) y \in \mathcal{N}^{[e-1]}$, and thus
$$0 = <x, [\pi]y - \sigma_e(\pi)y> = <[\pi] x - \sigma_e(\pi) x, y>$$
One then gets that $[\pi] x - \sigma_e(\pi)x \in \mathcal{N}^\bot = \mathcal{N}$. The results for the other sheaves are similar.
\end{proof}

One would of course want that this filtration coincides with the previous one. This is possible only if $d_i=g$ for all $1 \leq i \leq e$, which we will assume in the rest of the section.

\begin{defin}
One says that the filtration $\mathcal N^{[\bullet]}$ is compatible with the pairing if 
$$\mathcal{N}^{[2e-\ell]} = {\mathcal{N}^{[\ell]}}^\bot$$
for all $1 \leq \ell \leq e-1$.
\end{defin}

Let us be a little more explicit about the above condition. 
If $R$ is a polynomial, we denote by $\mathcal{E} [R]$ the kernel of $R(\pi)$ acting on $\mathcal{E}$. One sees in particular that $\mathcal{N}^{[\ell]} \subset \mathcal{E} [Q_\ell]$ for $1 \leq \ell \leq e$.

\begin{prop}
One has for $1 \leq \ell \leq e$
$$\mathcal{E} [Q_\ell]^\bot = \mathcal{E} [Q^\ell]$$
\end{prop}

\begin{proof}
Note that one has $\mathcal{E} [Q_\ell] = Q^\ell (\pi) \mathcal{E}$. The fact that $x$ belongs to $\mathcal{E} [Q_\ell]^\bot$ is thus equivalent to the fact that $<x, Q^\ell(\pi) y> = 0$ for all $y \in \mathcal{E}$. This is equivalent to the relation $Q^\ell (\pi) x = 0$.
\end{proof}

Since the multiplication by $Q^\ell (\pi)$ induces an isomorphism $\mathcal{E} / \mathcal{E}  [Q^\ell] \simeq \mathcal{E} [Q_\ell]$, one has an induced perfect pairing
$$h_\ell : \mathcal{E} [Q_\ell] \times \mathcal{E} [Q_\ell] \to \mathcal{O}_S $$
Explicitly, since $\mathcal{E} [Q_\ell] = Q^\ell(\pi) \mathcal{E}$, one has
$$h_\ell (Q_\ell (\pi) x, Q_\ell (\pi) y) = <x, Q_\ell(\pi) y> = <Q_\ell (\pi) x , y>$$

\begin{cor}
\label{proppolQ}
The filtration $\mathcal N^{[\bullet]}$ is compatible with the pairing if and only if $\mathcal{N}^{[\ell]}$ is totally isotropic in $\mathcal{E} [Q_\ell]$ for the pairing $h_\ell$, for every $1 \leq \ell \leq e$.
\end{cor}

\begin{proof}
To say that the filtration is compatible with the pairing amounts to say that for every $1 \leq \ell \leq e$, one has ${\mathcal{N}^{[\ell]}}^\bot = (Q^\ell(\pi))^{-1}  \mathcal{N}^{[\ell]}$. Since the orthogonal of $\mathcal{N}^{[\ell]}$ for $h_\ell$ is $Q^\ell(\pi) {\mathcal{N}^{[\ell]}}^\bot$, the result follows.
\end{proof}

\begin{rema} \label{isot}
In special fiber, the situation is again quite simpler. In this case, one has simply $\mathcal E [Q_{\ell}] = \mathcal E[\pi^\ell] = \pi^{e-\ell} \mathcal E $. The pairing $h_{\ell}$ on this sheaf is given by
$$ h_{ \ell} (\pi^{e - \ell} x, \pi^{e - \ell} y) = <\pi^{e - \ell} x , y> = <x, \pi^{e - \ell} y> $$
If $\mathcal F \subset \mathcal E[\pi^\ell]$ is totally maximally isotropic for $h_{ \ell}$, then its orthogonal in $\mathcal E$ is equal to
$$\mathcal F^\bot = (\pi^{e-\ell})^{-1} \mathcal F. $$
\end{rema}

\subsubsection{Application to $p$-divisible groups} 

Let $G \fleche S$ be a $p$-divisible group of height $h [F : \QQ_p]$, endowed with an $\mathcal O_L$-action. Thus, we can decompose $\omega_G$, a locally free $\mathcal O_S$-module into
\[ \omega_G = \bigoplus_{ \tau \in \mathcal T} \omega_{G,\tau}.\]
Assume that $\omega_{G,\tau}$ is locally free of rank $p_\tau$, and suppose given for all $\tau$ integers
\[ d_\tau^1,\dots,d_\tau^e,\]
such that $d_\tau^i \leq h$ for all $\tau,i$, and $d_\tau^1 + \dots + d_\tau^e = p_\tau$ for all $\tau$. Denote $f = [L^{ur}:\QQ_p]$,so that $\Ht(G) = efh$.
Define $\mathcal H := \mathcal H(G) := H^1_{dR}(G/S) := \mathbb D(G)_{S \fleche S}$ the evaluation of the crystal of $G$ (\cite{BBM}) on $S$. This is a locally free $\mathcal O_S \otimes_{\ZZ_p} \mathcal O_L$-module of rank $h$, which moreover splits as
\[ \mathcal H = \bigoplus_{\tau \in \mathcal T} \mathcal H_\tau,\]
and for each piece, there is an exact sequence given by the Hodge filtration,
\[ 0 \fleche \omega_{G,\tau} \fleche \mathcal H_\tau \fleche \omega_{G^D,\tau}^\vee \fleche 0.\]

\begin{defin}
\label{defPRG}
A Pappas-Rapoport datum for $G$, with respect to $L,(\sigma_{\tau,j}),(d_\tau^{j})_{\tau \in \mathcal T,j}$, is the datum of, for all $\tau$, 
of a full Pappas-Rapoport datum for $(\mathcal H_\tau,\omega_{G,\tau})$; i.e.
a filtration by locally direct $\mathcal O_S$-factors
\[ 0 = \omega_{G,\tau}^{[0]} \subset \omega_{G,\tau}^{[1]}  \subset \dots \subset \omega_{G,\tau}^{[e-1]} \subset \omega_{G,\tau}^{[e]} = \omega_{G,\tau} \subset \omega_{G,\tau}^{[e+1]} \subset \dots \subset \omega_{G,\tau}^{[2e-1]}  \subset\omega_{G,\tau}^{[2e]} = \mathcal H_\tau,\]

satisfying
\begin{enumerate}
\item For $j = 1,\dots,e$, $\dim_{\mathcal O_S} (\omega_{G,\tau}^{[j]}/\omega_{G,\tau}^{[j-1]}) = d_\tau^j$,
\item For $j = 1,\dots,e$ $\dim_{\mathcal O_S} \omega_{G,\tau}^{[e+j]}/\omega_{G,\tau}^{[e+j-1]} = h- d_\tau^{e-j+1}$
\item 
For $j = 1,\dots,e$, $([\pi] - \sigma_{\tau,j}(\pi) (\omega_{G,\tau}^{[j]}) \subset \omega_{G,\tau}^{[j-1]}$,
\item For $j = 1,\dots,e$, $  ([\pi]-\sigma_{\tau,e}(\pi))\dots([\pi] - \sigma_{\tau,e+j-1}(\pi)) \omega_{G,\tau}^{[e+j]} \subset \omega_{G,\tau}^{[e-j]}.$
\item For $j = 1,\dots,e$, $ \omega_{G,\tau}^{[e+j]} =(Q^{e-j} (\pi)  )^{-1} \omega_{G,\tau}^{[e-j]}.$
\end{enumerate}
\end{defin}

\begin{defin}
\label{def29}
Let $H$ be a locally free $R$-module (of finite rank), $R$ a ring, and denote
\[ H \otimes H^\vee \fleche R,\]
the perfect pairing between $H$ and $H^\vee = \Hom_R(H,R)$. Let $F \subset H$ be a locally direct factor. The association
\[ F \longmapsto F^\perp := \Ker(H^\vee \twoheadrightarrow F^\vee),\]
is an inclusion reversing involution between locally direct factors of $H$ and $H^\vee$.
\end{defin}
Thus by proposition \ref{propdual} we have

\begin{prop}
\label{propPRdual}
Let $(\omega_{G,\tau}^{[i]})_{i=0,\dots,2e}$ be a full Pappas-Rapoport datum for $G$, with respect to $L,(\sigma_{\tau,j})_{(\tau,j)},(d_\tau^j)_{(\tau,j)}$.
Then $((\omega_{G,\tau}^{[2e-i]})^\perp)_{i=0,\dots,2e}$ induces a full Pappas-Rapoport datum for $G^D$, with respect to $L,(\sigma_{\tau,j})_{\tau,j},(h-d_\tau^j)_{(\tau,j)}$.
Here $(.)^\perp$ denotes the previous involution under the identification $\mathcal H(G^D) = \mathcal H(G)^\vee$ (which satisfies $\omega_{G}^\perp = \omega_{G^D}$).
\end{prop}

\begin{defin}
\label{def218}
Suppose we are given an extension $L/L^+$ of\footnote{We allow $L = L^+\times L^+$ } degree $\leq 2$, and let $s \in \Gal(F/F^+)$.
Suppose we are given a polarisation $\lambda : G \overset{\sim}{\fleche} (G^D)^{(s)}$. Then we say that a Pappas-Rapoport datum $\mathcal R$ for $G$ is compatible with $\lambda$, if under the isomorphism $\lambda : \mathcal H(G) \overset{\sim}{\fleche} \mathcal H(G^D)^{(s)}$, the datum $\mathcal R$ and $(\mathcal R^\perp)^{(s)}$ (given by Proposition \ref{propPRdual} and twist by $s$) coïncides. In particular, this implies that 
\[ d_{\tau,i}(\mathcal R) = h-d_{s(\tau),i}(\mathcal R).\]
\end{defin}

\begin{rema}
\label{remPRAR}
In the situation when primes ramifies further in $L/L^+$, the previous compatibility is unfortunately impossible to achieve (except possibly in reduced special fiber).
For example let $L^+ = \QQ_p$ and $L = \QQ_p[T]/(E(T))$ a quadratic extension in which $p = (\pi)^2$ ramifies. Denote by $c(\pi) = \overline \pi$ the conjugate uniformizer.
Thus a Pappas-Rapoport datum for $G/S$, $L$, $(\pi,\overline\pi)$ and $d \leq h$, is the datum of
\[ 0 \subset \omega^{[1]} \subset \omega_G \subset \mathcal F^{[1]} \subset \mathcal H,\]
such that $\mathcal F^{[1]} = (T-\overline{\pi})^{-1}\omega^{[1]}$. The associated datum of $(G^D)^{(s)}$ is
\[ 0 \subset (\mathcal F^{[1]})^{\perp,s} \subset \omega_{G^D}^{(s)} \subset (\omega^{[1]})^{\perp,(s)} \subset H^{\vee,(s)}.\]
But as $T$ acts on $\mathcal H/\mathcal F^{[1]}$ as $\pi$, and thus on $(\mathcal F^{[1]})^{\perp,s}$ as $\overline\pi$. Moreover $(\mathcal F^{[1]})^{\perp,c}$ is of rank $h-d_1$ when $\omega^{[1]}$ is of rank $d_1$. There is thus no chance that given an isomorphism
\[ \lambda : G \overset{\sim}{\fleche} G^{D,(s)},\]
our Pappas-Rapoport datum is $\lambda$-compatible (except if $\pi = \overline\pi$ on $S$, for example in special fiber, and $2d_1 = h$)... 
We will refer to this as the case (AR) in the rest of the text.
\end{rema}

\subsection{Pappas-Rapoport models}
\label{sect24}
Let $\mathcal O_B$ be a $\ZZ_{(p)}$-order in $B$, preserved by $\star$, such that its completion is a maximal $\ZZ_p$-order in $B\otimes_{\QQ} \QQ_p$, and $(L,<.,.>)$ be a PEL $\mathcal O_B$-lattice \todo{$\mathcal O$ pas défini $\mathcal O = \mathcal O_B$ $\Box$} (see \cite{Lan} Definition 1.2.1.3) such that $(L,<.,.>)\otimes_{\mathbb Z} \QQ = (V,<.,.>)$.\todo{produit tensoriel sur  $\mathbb{Z}$ ! $\Box$} 
Suppose moreover Hypothesis \ref{hyp1}, that is to say $B_{\QQ_p}$ is isomorphic to a product of matrix algebras over (necessarily finite) extensions of $\QQ_p$. Note that we do not 
suppose that the extensions are unramified. Now suppose that $p$ is a \textit{good} prime, in the sense that $p  \nmid [L^\sharp:L].$\footnote{this is less strong than Lan's definition \cite{Lan}, 1.4.1.1, as we actually want to define a moduli problem for primes ramified in $\mathcal O$, but we, as Lan, assume that $p$ does not contribute to the level (implicitly as there will be no level at $p$), and exclude factors of type $D$.} This assumption will remain in force during all this article.

Let $K$ be an extension of $E_\nu$ (with $E$ the relfex field) which contains $F_i^{gal}$ for all $i$. We will want to consider a moduli problem $X$ over $\mathcal O_K$ of 
associating to $S$ quintuples $(A,\lambda,i,\eta,\omega^{[\cdot]})$ up to $\ZZ_{(p)}^\times$-isogenies, where $\omega^{[\cdot]}$ will be a Pappas-Rapoport datum with respect to a combinatorial data $\mathcal C$ which we now explain.

First, suppose given a abelian scheme $A$ over $S$ (itself over $\mathcal O_K$) endowed with an action $i : \mathcal O_B \fleche End(A) \otimes_\ZZ (\ZZ_{(p)})_S$. Decompose $\omega_A$ as before so we get a collection $(\omega_{i})_{i=1,\dots,r}$, and now the action of $\mathcal O_{F_i^{ur}}$ on $\omega_i$ can be splitted (as $S$ is over $\mathcal O_K$) as,
\[ \omega_i = \bigoplus_{\tau \in \mathcal T_i} \omega_{i,\tau},\]
where $\mathcal T_i = \Hom(F_i^{ur},\CC_p)$. Unfortunately if $F_i$ is ramified, we can't further decompose the $\omega_{i,\tau}$ as we did over $\QQ_p$. 

Denote by $\Sigma_{i,\tau}$ the subset of $\mathcal T'_i := \Hom(F_i,\CC_p)$ of embedding $\tau'$ that induces $\tau$ when restricted to $F^{ur}_i$.
Let us denote, for all $i \in \{1,\dots,r\}$, $\pi_i$ a chosen uniformiser of $F_i/F_i^{ur}$, $Q_i$ a corresponding Eisenstein polynomial, and let us choose an ordering $\Sigma_{i,\tau} = \{\sigma_{i,\tau,1},\dots,\sigma_{i,\tau,e_i}\}$ for the elements of $\Sigma_{i,\tau}$, for all $i,\tau$, where $e_i = [F_i:F_i^{ur}]$, (it corresponds to an ordering of the conjugate roots of $\tau(Q_i)$). 
This induces a bijection
\[ \begin{array}{ccc}
  \{1,\dots,e_i\} 
& \overset{\sigma_\bullet}{\fleche}  & \Sigma_{i,\tau}  \\
 j & \longmapsto  & \sigma_{i,\tau,j}  
\end{array},
\]
and a numbering $(d_{i,\tau,j})_{j = 1,\dots,e_i}$ such that $\{ d_{i,\tau,j} : j = 1,\dots,e_i\} = \{ d_{i,\tau'} : \tau' \in \Sigma_{i,\tau}\}$, by setting 
\[ d_{i,\tau,j} = d_{i,\sigma_{i,\tau,j}},\]
where $d_{i,\sigma_{i,\tau,j}}$ is defined in section \ref{sect2.2}. We suppose moreover, to reflect remark \ref{remadpol}, that the choice of these bijections \todo{$s$ pas défini ? A préciser $\Box$} implies that, for all $i,\tau,j$,
\[ d_{i,\tau,j} = h_i-d_{s(i,\tau),j},\]
where as before $s$ is the action induced by $\star$, the involution on $B$, on $\prod_i \Hom(F_i,\CC_p)$.
To ease the notation, we will write $$\mathcal{C} = ((\sigma_{i,\tau,j})_{i,\tau,j},(d_{i,\tau,j})_{i,\tau,j})$$ and $\mathcal{C}_i= ((\sigma_{i,\tau,j})_{\tau,j},(d_{i,\tau,j})_{\tau,j})$ for every $i$. One will also write $\mathcal{C}^D = ((\sigma_{i,\tau,j} \circ s)_{i,\tau,j},(h_i - d_{i,\tau,j})_{i,\tau,j})$ and $\mathcal{C}^D_i= ((\sigma_{i,\tau,j} \circ s )_{\tau,j},h_i -(d_{i,\tau,j})_{\tau,j})$.
We can moreover, if we denote $\mathcal H = H^1_{dR}(A/S)$, decompose
\[ \mathcal H = \bigoplus_{i} \mathcal H_{i},\]
with $\mathcal H_i$ a $M_{n_i}(\mathcal O_{F_i})$-module, corresponding by Morita equivalence to $e_i\mathcal H_i$, which we can further decompose
\[ e_i\mathcal H_i = \bigoplus_{\tau : F_i^{ur} \fleche \overline\QQ_p} \mathcal H_{i,\tau}.\]

Remark that if $S$ is actually over $\QQ_p$, then,
\[ \omega_{i,\tau} = \bigoplus_{\tau' \in \Sigma_{i,\tau}} \omega_{i,\tau'},\]
and using the previous decomposition for $\omega$, we get $\omega_{i,\tau}^{[\cdot]}$ a filtration of $\omega_i$, by,
\[ \omega_{i,\tau}^{[j]} = \bigoplus_{r = 1}^j \omega_{i,\sigma_{i,\tau,r}}.\]

\begin{example}
As explained in \cite{KotJams}, when $C$ is a semi-simple algebra over a field (say of characteristic $p$), then a $C$-module $V$ is determined by its determinant $\det_V$ associated by Kottwitz. Unfortunately in the ramified case $\mathcal O_B\otimes \FP$ is no longer semi-simple and the determinant fails to determine $V$. For example say
$B_{\QQ_p} = F$ a totally ramified extension of $\QQ_p$ of degree $e$, and thus $\mathcal O_B \otimes \FP = \FP[X]/(X^e)$. Let $V = \FP[X]/(X^e)$ and $W = \FP^e$.
Then $\det_V(T_1,\dots,T_e) = \det_W(T_1,\dots,T_e) = T_1^e$ but obviously $V$ and $W$ are not isomorphic as $\FP[X]/(X^e)$-modules.
\end{example}


Now suppose that $\lambda$ is a $\ZZ_p\times$-polarisation on $A$, and that $i$ is a $\mathcal O_B$ structure for $(A,\lambda)$ (i.e. satisfies the Rosatti Condition, cf \cite{Lan} Definition 1.3.3.1). Using this $\lambda$ we deduce an isomorphism
\[ e_i\mathcal H_i \simeq e_{s(i)}\mathcal H_{s(i)}^\vee,\]
coming from an isomorphism of the associated $p$-divisible group when decomposing $\mathcal A[p^\infty]$, and we can thus apply the Definition \ref{def218} for $F_i/F_i^+ = F_i^{\star = 1}$ (when $s(i) = i$) or $F_i \times F_{s(i)}$ when $i \neq s(i)$ (but beware of the case (AR) as explained in Remark \ref{remPRAR}).

\begin{defin}[\cite{P-R} Sections 14 and 9]
\label{defPRmoduli}
Let $S$ be a $\mathcal O_K$-scheme. A Pappas-Rapoport datum for $(A,\lambda,\iota)/S$, with respect to $\mathcal C$ is the datum, for every $(i,\tau)$, of
\begin{itemize}
\item  if $i$ does not fall in case AR\footnote{This will be explained later, here by falling in case AR we just mean that $F_i$ is stable by $\star$ and the extension $F_i/F_i^{\star = 1}$ is ramified.}, a full Pappas-Rapoport datum for $(\mathcal H_{i,\tau},\omega_{i,\tau})$ 
associated to
$\mathcal{C}_i$, as in definition \ref{defPRG},
 that is moreover $\lambda$-compatible with $\mathcal C_{s(i)}^D$ (see definition \ref{def218}), \todo{Précision sur ce que ça veut dire ?}
 \item if $i$ falls in case AR\footnote{As explain in Remark \ref{remPRAR}, we cannot ask for $\lambda$-compatibility here. Unfortunately this moduli space will not be studied in much details here, but we will show that our two main theorems fail in this AR case.}, we ask 
for a full Pappas-Rapoport datum $\mathcal R$ for $(\mathcal H_{i,\tau},\omega_{i,\tau})$ with respect to $\mathcal{C}_i$, 
(this automatically induces a full Pappas-Rapoport datum $\mathcal R^{\perp,(s)}$ for $(\mathcal H_{i,\tau}^{D,s},\omega_{A^D,i,\tau)}^{(s)}) = (\mathcal H_{i,\tau},\omega_{A,i,\tau)})$ with respect to 
$\mathcal{C}_i^D$).
\end{itemize}
%
\end{defin}

Consider the moduli problem $X$ over $\mathcal O_K$, associating to $S$ quintuples $(A,\lambda,i,\eta,\omega^{[\cdot]})$ up to $\ZZ_{(p)}^\times$-isogenies, where 
\begin{itemize}
\item $A \fleche S$ is an abelian scheme,
\item $\lambda : A \fleche {^t}A$ is a $\ZZ_{(p)}^\times$-polarisation,
\item $i : \mathcal O_B \fleche End(A) \otimes_\ZZ (\ZZ_{(p)})_S$ is a $\mathcal O_B$-structure of $(A,\lambda)$ (in particular $\star$ induces the Rosatti involution)
\item $\eta$ is a rational level structure outside $p$  (see \cite{Lan} section 1.4.1)
\item $\omega^{[\cdot]}$ is a Pappas-Rapoport datum for $\mathcal C$, which is defined in Definition \ref{defPRmoduli}. \todo{Pas très claire, la définition arrive bien après $\Box$}
\end{itemize}

\begin{prop} The moduli space $X$ associating to each $S$ over $\Spec(\mathcal O_K)$ the set of isomorphism classes of quintuples $(A,\lambda,i,\eta,\omega^{[\cdot]})$ is representable by a quasi-projective scheme.
\end{prop}

\begin{proof}
This is shown in a local context in \cite{P-R} : the morphism who forgets the Pappas-Rapoport datum is relatively representable over the (PEL) moduli space (Kottwitz's model).
Thus $X$ is representable, locally fibrered over Kottwitz's model as a closed subset of a product of Grassmanian (equivalently it is a closed subset of some partial flag variety for $\omega_{A^{univ}}$ over Kottwitz's model).
\end{proof}

\begin{prop}
Let $K^p \subset G(\mathbb A_f^p)$ be a compact open subgroup as before (neat). Let $\mathfrak C \subset G(\QQ_p)$ be the stabilizer of $L$, and consider the compact open $\mathfrak CK^p$. Let us choose $K$ a $p$-adic completion of the Galois closure of the $F_i$ as before. Then the Pappas-Rapoport model $X/\Spec(\mathcal O_K)$ coincides with $\mathcal S_{\mathfrak CK^p}$ over $K$, i.e. $X$ is an integral model of $\mathcal S_{\mathfrak CK^p}$ over $\mathcal O_{K}$.
\end{prop}

\begin{proof}
Obviously if $(A,i,\lambda,\eta)$ is a quadruple in $\mathcal S_{\mathfrak CK^p}(S)$, where $S$ is over $K$, there is a canonical filtration of $\omega_A$, as explained before, as we have fixed the bijections $\sigma_\bullet$.
Moreover, each quintuple $(A,i,\lambda,\eta,\omega^{[\cdot]})$ satisfies the Kottwitz's determinant condition as the filtration given by $\omega^{[\cdot]}$ on $\omega_{i,\tau}$ is split, and the dimensions are fixed by the Pappas-Rapoport condition. The equivalence beetwen definition by $\ZZ_{(p)}^\times$-isogeny classes and quasi-isogeny classes (in characteristic 0) is then \cite{Lan} Proposition 1.4.3.4.
\end{proof}

From now on fix a level $K^p \subset G(\mathbb A_f^p)$ outside $p$  and $\mathfrak C$ as before at $p$ ("without level at $p$"), call $X$ the Pappas-Rapoport model over $\mathcal O_{K}$
of the Shimura variety $\mathcal S_{\mathfrak CK^p}$. It thus make sense to reduce $X$ over $\kappa$, the residue field of $K$. The goal of this article is to study the geometry of $X_\kappa := X \times \Spec(\kappa)$.

\subsection{Polygons}
\label{sectdecpol}
As explained in the previous section, we can decompose the lie algebra $\omega_A$ of the universal abelian scheme $A$ over $X$ through the action of $\mathcal O_B\otimes \ZZ_p$. Actually, we can also decompose the $p$-divisible group $A[p^\infty]$. According to Hypothesis \ref{hyp1}, we write,
\[ \mathcal O_B \otimes \ZZ_p = \prod_{\pi \in \mathcal P} M_{n_\pi}(R_\pi),\]
where $\pi \in \mathcal P$ is a new indexation for $i \in \{1,\dots,r\}$ where the two factors $i, s(i)$, exchanged by $\star$, share the same index $\pi$, and, \todo{pas très clair $\Box$ ?}
$R_\pi = \mathcal O_{F_\pi}$ if $i = s(i)$, or $R_\pi = \mathcal O_{F_\pi} \times \mathcal O_{F_\pi}$ if there are two factors (i.e. $R_\pi = \mathcal O_{F_i} \times \mathcal O_{F_{s(i)}}$ if $\pi = [i]$ when $i \neq s(i)$). Thus $\mathcal P = \{1,\dots,r\}_{/\sim_s}$.
We refer to the $\pi \in \mathcal P$ as places over $p$ in $B$. We can thus decompose, 
\[ A[p^\infty] = \prod_{\pi \in \mathcal P} A[\pi^\infty],\]
where $A[\pi^\infty]$ is a $M_{n_\pi}(R_\pi)$-module $p$-divisible, and by Morita equivalence, 
\[ A[\pi^\infty] = \mathcal O_{R_\pi}^{n_\pi} \otimes_{\mathcal O_{R_\pi}} G_\pi.\]
Note that because $\lambda$, the (universal) polarisation of $A$ is compatible with $\star$, each factor $A[\pi^\infty]$ is still endowed with a polarisation $\lambda_\pi$ and thus also $G_\pi$.
We will use this decomposition of $A[p^\infty]$ all the time as, if we know the $n_\pi$, it is equivalent to know $A[p^\infty]$ (and $\lambda$) or the collection 
$(G_\pi)_\pi$ (and the $\lambda_\pi$).

For all $\pi$, $G_\pi$ is a polarized $p$-divisible group over $X$, the Pappas-Rapoport model, endowed with an action of $R_\pi$. 

Moreover, if $R_\pi = \mathcal O_{F_\pi} \times \mathcal O_{F_\pi}$ and $\star$ exchanges the two factors, we can further decompose $G_\pi = H_\pi \times H_\pi^D$ and 
$\lambda_\pi$ exchanges the two factors. In this case, called (Split) or (AL), the datum of $(G_\pi,\lambda_\pi)$ is equivalent to $H_\pi$.

Otherwise $G_\pi$ is a $p$-divisible $\mathcal O_{F_\pi}$-module with a polarisation $\lambda_\pi$ such that either,
\begin{itemize}
\item $\star$ induces the identity on $\mathcal O_{F_\pi}$, which is equivalent for $\lambda_\pi$ to be compatible with the $\mathcal O_{F_\pi}$-action, which we refer to as case (C).
\item $\star$ is an automorphism of order 2 of $\mathcal O_{F_\pi}$, and denote $\mathcal O_{F^+_\pi}$ the subfield fixed by $\star$. If $\mathcal O_{F_\pi}$ is unramified over 
$\mathcal O_{F_\pi^+}$, we refer to this case as (Inert) or (AU), and if the extension is ramified, as (Ram) or (AR). In this two cases $G_\pi$ satisfies the symmetry,
\[ G_\pi^D \overset{\lambda_\pi}{\simeq} G^{(c)}_\pi,\]
where $c$ is the order-two automorphism of $F_\pi$ induced by $\star$, and $G^{(c)}_\pi$ is the $p$-divisible $\mathcal O_{F_\pi}$-module where the endomorphism structure is $\iota^{(c)} := \iota \circ c$ if $\iota$ denotes the endomorphism structure of $G_\pi$.
\end{itemize}

\begin{rema}
The previous denomination comes from the possible decompositions of the $p$-divisible group of an abelian variety endowed with an action of the ring of integer of a totally real field (C), or a CM-field $F/F^+$ in which a place $\pi$ of $F^+$ splits in $F$ (AL), is inert in $F$ (AU) or is ramified in $F$ (AR), which itself is related to the classification of Lie algebras, symplectic (C) or unitary (A) (as we have excluded orthogonal factors (D)).
\end{rema}

From now on, we will fix a element $\pi \in \mathcal P$. For the rest of this subsection suppose our base scheme $S$ is a field $k$ over $\kappa$ (thus of characteristic $p$). 
Let us be more explicit about the different cases.

\subsubsection{Case C}

In the case $C$, we will denote the $p$-divisible group $G_\pi$ simply by $G$, and $F_\pi$ by $L$\footnote{We had also denoted $L$ the lattice in the definition of $X$, but as it will not appear anymore we hope it won't create any confusion...}. The $p$-divisible group $G$ has an action of $\mathcal O_L$, and a polarization. It has height $2dg$ and dimension $dg$, where $d$ is the degree of $L$ over $\mathbb{Q}_p$.  The sheaf $\omega_G$ decomposes as 
$$\omega_G = \bigoplus_{\tau \in \mathcal{T}} \omega_\tau$$
where $\mathcal T$ is the set of embeddings of $L^{ur}$.
Recall the Hodge filtration for $G$
$$0 \to \omega_G \to H^1_{dR} \to \omega_{G^D}^\vee \to 0,$$
where $G^D$ is the Cartier dual of $G$. This exact sequence splits according to the elements of $\mathcal{T}$. The PR condition for $G$ is then as follows. \\
For each $\tau \in \mathcal{T}$, one has a filtration 
$$\omega_\tau^{[0]} = 0 \subset \omega_\tau^{[1]} \subset \dots \subset \omega_\tau^{[e]} = \omega_\tau,$$
where $\omega_\tau^{[i]}$ is locally a direct factor of rank $gi$. Moreover, one has the following compatibility with the polarization:
$${\omega_\tau^{[i]}}^\bot = (\pi^{e-i})^{-1} \omega_\tau^{[i]}$$
this equality being taken in $H^1_{dR,\tau}$ for $1 \leq i \leq e$ (recall $\pi_i = 0$ in $S$ here).\\
For each $\tau \in \mathcal T$, one can define the polygon $Hdg_\tau(G)$; it is defined thanks to $\omega_\tau$ as in \cite{BH} Definition 1.1.7. It starts at $(0,0)$ and ends at $(2g,g)$. Since $G$ has a polarization, this polygon is symmetric: its slopes are $\lambda_1, \dots, \lambda_g, 1-\lambda_g, \dots, 1 - \lambda_1$. We define the polygon $Hdg(G)$ as the mean of the polygons $Hdg_\tau(G)$. \\
The polygons $PR_\tau$ and $PR$ are all equal: the have slope $0$ and $1$, each of them with multiplicity $g$. \\
We define the Newton polygon of $G$ as in \cite{BH} Def. 1.1.8 , and denote it $Newt(G)$; it is also symmetric.

\begin{prop}[\cite{BH} Théorème 1.3.1]   \todo{Amléliorer les références $\Box$}
 \label{prop218}
One has the inequalities
$$Newt(G) \geq Hdg(G) \geq PR,$$
and these polygons are all symmetric.
\end{prop}

\subsubsection{Case AL}

In this case, one has $G_\pi = H_\pi \times H_\pi^D$. We will consider the $p$-divisible group $G = H_\pi$. It is endowed with an action of $\mathcal{O}_L$ but has no polarization. The sheaf $\omega_G$ decomposes as
$$\omega_G = \bigoplus_{\tau \in \mathcal{T}} \omega_\tau$$
Fix 
$(a_{\tau,j}) \in \ZZ^{\mathcal T \times\{1,...,e\}}$, where $e$ is the ramification index of $L$. Denote by $a_\tau = \dim \omega_\tau$ and $b_\tau = \dim \omega_{G^D,\tau}$. Then $h' = a_\tau + b_\tau$ is independant of $\tau$. In the global setting, $(a_{\tau,j})$ will coincide with the part of the integers $d_{i,\tau,j}$ corresponding to $H_\pi$. The PR datum (for $(a_{\tau,j})$) is then as follows. For each $\tau \in \mathcal{T}$, one has a filtration 
$$\omega_\tau^{[0]} = 0 \subset \omega_\tau^{[1]} \subset \dots \subset \omega_\tau^{[e]} = \omega_\tau$$
where $\omega_\tau^{[i]}$ is locally a direct factor of rank $a_{\tau,1} + \dots + a_{\tau,i}$. \\
Note that this PR datum induces a PR datum for $G^D$ (see Definition \ref{def29}).
For each $\tau \in \mathcal T$, one can define the polygon $Hdg_\tau(G)$; it is defined thanks to $\omega_\tau$, see \cite{BH} Definition 1.1.7. It starts at $(0,0)$ and ends at $(a_\tau+b_\tau,a_{\tau,1} + \dots + a_{\tau,e}) = (h',a_\tau)$. We define the polygon $Hdg(G)$ as the mean of the polygons $Hdg_\tau(G)$. \\
The polygons $PR_\tau$ is defined in \cite{BH} section 1. The polygon PR is the mean of the polygons $PR_\tau$. \\
We define the Newton polygon of $G$ using \cite{BH} def. 1.1.8, and denote it $Newt(G)$. As we have used $H_\pi$ instead of $H_\pi \times H_\pi^D$, these polygons do not need to be polarized (i.e. symmetric in any sense). For the following proposition, we refer again to \cite{BH} Théorème 1.3.1.

\begin{prop}
 \label{prop219}
One has the inequalities
$$Newt(G) \geq Hdg(G) \geq PR.$$
\end{prop}

\subsubsection{Case AU}

In this case, we define $G = G_\pi$, we denote $F_\pi$ by $L$ and by $L^+$ the subfield of element fixed by $\star$, which we denote by $\overline{.}$ as a conjugation. It is endowed with an action of $\mathcal O_L$ but has also a polarization. The sheaf $\omega_G$ decomposes as
$$\omega_G = \bigoplus_{\tau \in \mathcal{T^0}} \omega_\tau \oplus \omega_{\overline{\tau}}$$
where $\mathcal{T}$ is the set of embedings of $F$ and $\mathcal{T^0}$ the embeddings of $F$ modulo conjugation. We define $(a_{\tau,j}) \in \ZZ^{\mathcal T \times\{1,...,e\}}$ as before, and $a_\tau = \dim \omega_\tau$, $b_\tau = \dim \omega_{G^D,\tau} = \dim \omega_{\overline\tau} = a_{\overline\tau}$, and $h' = a_\tau + b_\tau$.
The PR condition is then as follows. For each $\tau \in \mathcal{T}$, one has a filtration 
$$\omega_\tau^{[0]} = 0 \subset \omega_\tau^{[1]} \subset \dots \subset \omega_\tau^{[e]} = \omega_\tau$$
where $\omega_\tau^{[j]}$ is locally a direct factor of rank $a_{\tau,1} + \dots + a_{\tau,j}$. \\
Note that this PR condition coincides with the induced PR condition for $\omega_{\overline{\tau}}$ thanks to the compatibility with the polarization.
For each $\tau \in \mathcal T$, one can define the polygon $Hdg_\tau(G)$; it is defined using $\omega_\tau$. It starts at $(0,0)$ and ends at $(a_\tau+b_\tau,a_{\tau,1} + \dots + a_{\tau,e})=(h',a_\tau)$. 
We define the polygon $Hdg(G)$ as the mean of the polygons $Hdg_\tau(G)$ for all $\tau \in \mathcal T$.\\
The polygons $PR_\tau$ is defined in \cite{BH}. The polygon PR is the mean of the polygons $PR_\tau$. \\
We define the Newton polygon of $G$ as $Newt(G)$. Once again, thanks to \cite{BH} Théorème 1.3.1. one has the following result. 

\begin{prop}
 \label{prop220}
One has the inequalities
$$Newt(G) \geq Hdg(G) \geq PR,$$
and these polygons are symmetric.
\end{prop}

\subsubsection{Case AR}

In this case also we still define $G = G_\pi$ which is still polarized and carries an action as in case C or AU, but we no longer have the $\lambda$-compatibility for the Pappas-Rapoport datum. This doesn't change anything regarding the polygons : we can still define Hodge and Newton polygons using the action as in \cite{BH} Section 1 (this doesn't use the Pappas-Rapoport datum), and the polarisation implies that $\Newt(G) = \Newt(G^D)$ and $\Hdg(G) = \Hdg(G^D)$ (equalities between $\Hdg_\tau(G) = \Hdg_{\overline\tau}(G^D)$). Moreover the Pappas-Rapoport polygon only depends on the integers $(d_{\tau,j})_{\tau,j}$, and these are symmetric for $G$ and $G^D$ (see remark \ref{remadpol}).
 In particular, thanks to \cite{BH} Théorème 1.3.1. one has the following result.
 
 \begin{prop}
 \label{prop221}
One has the inequalities
$$Newt(G) \geq Hdg(G) \geq PR,$$
and these polygons are symmetric.
\end{prop}

\subsection{Stratifications}

Using the previous polygons, we can define subsets of the reduction of $X$ modulo $p$, $|X_{\kappa}|$. 
Denote by $Pol$ the set of polygonal convex functions on $[0,\dots,h]$ with breakpoints at abscissas in $\frac{1}{e}\ZZ$. The previous polygons define two maps,
\[
|X_\kappa| \overset{\Newt_\pi}{\fleche} Pol \quad \text{and} \quad |X_\kappa| \overset{\Hdg_\pi}{\fleche} Pol.\]

\begin{prop}
 \label{prop222}
The maps $\Newt_\pi$ and $\Hdg_{\pi,\tau}$ are semi-continuous, in the sense that polygons can only descend by generisation. Moreover they have same begining and end-points (which are always locally constant and constant in our global situation).
\end{prop}

\begin{proof}
The result on the Newton polygon is well-known, see for example \cite{RR} Theorem 3.6. For the Hodge polygon, note that locally on $X_\kappa$ we can trivialise 
$\omega_{G_\pi,\tau}$ and the action of $\pi$ on it is nilpotent (as $\pi^e = p = 0$ on $\mathcal O_{X_\kappa}$). Thus there is (Zariski locally) a continuous map
\[ X_\kappa \fleche Nilp_{p_\tau},\]
to the Nilpotent cone of $\GL_{p_\tau}$, sending a point to the matrix of $\pi$. We can check that the Hodge stratas are exactly the pullback of the stratification on the nilpotent cone. But now the analogous result is known for $Nilp_{p_\tau}$.
\end{proof}

There is moreover a (constant) map $PR_\pi : |X_\kappa| \fleche Pol$. If $\pi$ is understood from the context, we will drop it from the previous notations.
Recall the following,

\begin{defin}
Define the Newton stratification of the reduction mod $p$ of the Pappas-Rapoport model $X \otimes_{\mathcal O_K} \overline{\FP} = \coprod_\nu X^\nu$, by,
\[ X^\nu = \{ x \in X_{\kappa} : \Newt(x) = \nu\}.\]
The locus $X^{\nu =PR}$ of points $x \in |X_\kappa|$ such that $\Newt(x) = \PR(x)$ is called the $\mu$-ordinary locus (for $\pi$). 
\end{defin}

In particular the $\mu$-ordinary locus is an open stratum, by Propositions \ref{prop222} and \ref{prop218},\ref{prop219},\ref{prop220},\ref{prop221}. 
There is another natural stratification with another natural open stratum :

\begin{defin}
The locus $X_{\nu = PR}$ where $\Hdg_\pi(x) = \PR(x)$ is called the generalised Rapoport locus (for $\pi$). It contains the $\mu$-ordinary locus because of the inequalities recalled in the previous section. 
More generally we can define the Hodge stratification, $X_\kappa = \coprod_{\nu} X_\nu$, by
\[ X_{(\nu_\tau)} = \{ x \in X_\kappa : \Hdg_{\pi,\tau}(x) = \nu_\tau, \forall \tau\}.\]
\end{defin}

As for every $\tau$, we have $\Hdg_{\pi,\tau} \geq PR_\tau$, we have $\Hdg_\pi = \PR$ if and only if $\Hdg_{\pi,\tau} = \PR_\tau$ for all $\tau$. Remark also that in case $(C),(AU),(AR)$, because of the polarisation there is a symmetry between $\Hdg_{\pi,\tau}$ and $\Hdg_{\pi,\overline\tau}$. In particular we should only consider symmetric data $(\nu_\tau)$ in these cases in order to have non empty strata.

\begin{example}
\begin{enumerate}
\item If the PEL datum is unramified (or if $\pi$ is unramified) the Hodge polygon $\Hdg_\pi$ is constant on $X$, thus the generalised Rapoport locus consists in all the variety in this case.
\item In the Hilbert-Siegel case, the generalised Rapoport locus is the Rapoport locus, i.e. the locus where the conormal sheaf $\omega_{G_\pi}$ is locally free as 
$\mathcal O_{X_\kappa} \otimes_{\ZZ_p} \mathcal O_F$-module. It thus contains the $\mu$-ordinary locus (which is just the ordinary locus in this case).
\end{enumerate}

\end{example}

\subsection{Smoothness of the integral model}

In this section we prove that the Pappas-Rapoport model $X$ is smooth if all the primes above $p$ falls in type (AL),(AU) or (C) (i.e. doesn't fall in case (AR)).
We will reduce to work locally, thus let $G/S$ be a $p$-divisible group over a scheme $S$, endowed with an action of $\mathcal O_{L}$, where $L/\QQ_p$ is a finite extension\footnote{$G$ will be $G_\pi$ or $H_\pi$ and $L$ will be the corresponding $F_\pi$ as in the previous section}, possibly a polarisation, and a Pappas-Rapoport datum. 
We call such a $p$-divisible group, with action, eventual polarisation, and Pappas-Rapoport datum a $p$-divisible $\mathcal D$-module, with $\mathcal D$ referring to the type of the extra-datum (including the Pappas-Rapoport datum).
Denote by $\mathcal H = H^1_{dR}(G/S)$ the locally free $\mathcal O_S$-module associated to $G$. It is actually a locally free $\mathcal O_S \otimes_{\ZZ_p}\mathcal O_{F}$-module of rank $h$. It is endowed, except in case (AL), with a polarisation (i.e. a perfect pairing)
\[ < . ,.> : \mathcal H \times \mathcal H^{s} \fleche \mathcal O_S,\]
which is alternating, and such that
\[ <x,ay> = <s(a)x,y>, \quad \forall x,y \in \mathcal H, \forall a \in \mathcal O_{F}.\]

\begin{theor}
\label{thr229}
Suppose that for every prime $\pi$ above $p$, $\pi$ doesn't fall in case (AR).
Then $X$ is smooth.
\end{theor}

\begin{proof}
As $X$ is of finite presentation, it is enough to show it is formally smooth. Let $S \twoheadrightarrow R$ be a surjective morphism of rings with ideal $I$ such that $I^2 = 0$.
In particular $I$ is endowed with nilpotent divided powers, and thus we can use Grothendieck-Messing's theory. Thus let $x \in X(R)$. If $p$ is invertible on $R$, as we know that $X$ is smooth in generic fiber, there is $y$ in $X(S)$ above $x$ and we are done. Otherwise, by Serre-Tate, it is enough to lift the $p$-divisible group, and we can divide the task between the primes above $p$ in $\mathcal O_B$. Thus fix one such prime and $G/R$ the associated $p$-divisible group (with extra structures). By Grothendieck-Messing, it is enough to lift $\omega_G$ together with its Pappas-Rapoport datum as a locally direct factor (stable by $\mathcal O_F$ and totally isotropic) in
\[ \mathcal H_G \otimes_R S.\]
We will lift successively $\omega_G^{[1]}, \dots, \omega_G^{[e]}$. Thus fix $\tau$ an embedding of $L^{ur}$. We will work separately for each $\tau$. Recall that in case (C) or (AU) we have a polarisation
\[ \mathcal H_\tau \times \mathcal H_{\overline\tau} \fleche R,\]
that lifts to $S$ (as it is defined on the crystalline site of $R$).  In particular in case $(AU)$ it will be sufficient to choose one element in $[\tau] = \{\tau,\overline\tau\}$, say $\tau$, for each embedding $\tau$, lift the Pappas-Rapoport datum in $\mathcal H_\tau$ and take its orthogonal which will be a lift of the Pappas-Rapoport datum for $\overline\tau$. Thus this is similar to case (AL). In case (C), we will need the Pappas-Rapoport datum to be moreover totally isotropic. 

Recall that the sheaf $\mathcal H_\tau$ is locally free as a $R [T] / Q(T)$-module, where $Q$ is the Eisenstein polynomial of an uniformizer. We refer to sections \ref{sect232} and \ref{sectdecpol} for the notations. \todo{Rappeler les notations des sections précédentes} We have a submodule
\[ \omega_\tau^{[1]} \subset \mathcal H_\tau[T-\pi_1],\]
which is moreover totally isotropic for $h_{\tau,1}$ by Corollary \ref{proppolQ} in case (C). Thus, there exists a (totally isotropic, in case (C)) lift of this module to
\[ \mathcal H_\tau \otimes_R S[T-\pi_1]\]
that we denote by $\widetilde{\omega}_\tau^{[1]}$. 
Suppose that we have constructed for $1 \leq \ell \leq e-1$, locally direct factors
\[ \widetilde\omega_\tau^{[1]} \subset \dots \subset \widetilde\omega_\tau^{[\ell]} \subset \mathcal H_\tau[Q_\ell]\otimes_R S,\]
which are moreover isotropic for $h_{\tau,\ell}$ in case (C), lifting the previous datum over $R$ to $S$. 
Denote 
\[ E^\ell_\tau = \left(\mathcal H_\tau \otimes_R S 
 /\widetilde{\omega}_{\tau}^{[\ell]} \right) [Q^\ell],\]
It is a locally free $S[T]/(Q^\ell(T))$-module, which contains modulo $I$ the image of $\omega^{\ell + 1}_\tau$ 
as a locally direct factor.  In case (C), the fact that $\widetilde\omega_\tau^{[\ell]}$ is totally isotropic for $h_{\ell, \tau}$ implies that $E^\ell_\tau$ inherits the pairing from $\mathcal{H}_\tau$.
Thus, it is enough to lift the image of $\omega^{\ell + 1}_\tau$ 
as a locally direct factor in
\[ E_\tau^\ell[T-\pi_{\ell+1}],\]
which is moreover in case (C) totally isotropic for $h_{\tau,\ell+1}$ by Corollary \ref{proppolQ}. 
But such a lift exists (by smoothness of the appropriate partial flag variety), thus there exists $\widetilde{\omega}_\tau^{[\ell+1]} \subset \mathcal H_\tau \otimes S$ lifting ${\omega}_\tau^{[\ell+1]}$. 
By induction, we can thus find a lift of the filtration
\[ 0 \subset \omega^{[1]}_\tau \subset \dots \subset \omega^{[e]}_\tau,\]
to $S$ satisfying all the assumptions of the Pappas-Rapoport datum. Thus by Grothendieck-Messing (as $I^2 = 0$) there exists a point $y \in X(S)$ lifting $x$, and $X$ is smooth.
\end{proof}

\begin{rema}
As shown in the proof, we could have argued slightly differently using a local model for $X$ in the spirit of \cite{P-R}.
\end{rema}

\begin{rema}
Unfortunately the analogous result is not true in case (AR). It is possible to construct flat integral model in this case too, as done for example in \cite{Kramer} Theorem 4.5 for $U(1,n-1)$ for a ramified quadratic extension (a more general possible construction is in \cite{P-R} but the geometry isn't studied there). In \cite{Kramer}, is it constructed a local model (but the global constructed would lead to the same singularities) $\mathcal M$ which is regular (thus flat) and whose special fiber is the union of two smooth irreducible varieties of dimension $n-1$ crossing along a smooth irreducible variety of dimension $n-2$. See a calculation in \ref{AppE}.
\end{rema}

\subsection{Deformations and Displays}

In order to construct deformations 
of a point $x$ in $X_\kappa$, associated to a datum $(A,\lambda,i,\eta,(\omega^{[\cdot]}_\bullet))$, we will deform the $p$-divisible group $A[p^\infty]$, the action of $\mathcal O_B$, the polarisation and the Pappas-Rapoport condition, and use Serre-Tate theory (because $\eta$ is a level structure outside $p$, we can deform it trivially).

Moreover, using the previous simplification of $A[p^\infty]$ using $\mathcal O_B\otimes \ZZ_p$ and Morita equivalence, it is enough to deform the polarised $p$-divisible $\mathcal O_{F_\pi}$-modules $G_\pi$ together with their Pappas-Rapoport filtration, i.e. the $p$-divisible $\mathcal D_\pi$-module. To such a $p$-divisible group over a perfect field $k$ of characteristic $p$ is associated a Dieudonn\'e module over $W(k)$ (more precisely its Dieudonn\'e crystal)
and we want to deform it over $k[[X]]$ such that the special fiber at $X=0$ corresponds to $G$, and the generic fiber satisfies better properties, like being $\mu$-ordinary for example.
In order to do this, we will use the theory of displays (cf. \cite{Zink} and \cite{LauDieu} for the equivalence with etale part). We will be interested mainly in the tools developed in \cite{Wed1} section 3.2 (particularly 3.2.7 and theorem 3.2.8). In particular

\begin{prop}[Zink, Wedhorn, see \cite{Wed1} Theorem 3.2.8]
\label{Wed328}
Let $k$ be a perfect field of characteristic $p$, $G/k$ (with additional structures $\iota_0,\lambda_0$) and denote $P_0$ the associated display (with additional structures $\iota_0,\lambda_0$). Let $N$ be a $W(k)$-linear endomorphism of $P_0$, satisfying,
\begin{enumerate}
\item $N^2 = 0$
\item $N$ is skew-symmetric with respect to $\lambda_0$
\item $N$ is $\mathcal O$-linear
\end{enumerate}
Then there exists a deformation $(P,\iota,\lambda)$ of $(P_0,\iota_0,\lambda_0)$ (of display with additional structures) over $k[[t]]$ whose associated $p$-divisible group (with additionnal structure) $(X_N,\iota,\lambda)$ lifts $(X,\iota_0,\lambda_0)$ and such that, if $P_0$ is bi-infinitesimal,
\[ (X_N,\iota,\lambda) \otimes_{k[[t]]} k((t))^{perf} \simeq \mathcal{BT'}((P,\iota,\lambda)\otimes_{k[[t]]} k((t))^{perf}),\]
where $\mathcal{BT}'$ associate to a crystal (over a perfect field) its $p$-divisible group.
\end{prop}

\begin{rema}
Conditions 1. and 2. are only needed to lift the polarisation. In particular, they will not be needed in case (AL).
We will use this kind of deformation only to modify the Newton polygon, in particular we will be able to choose any lift of the Pappas-Rapoport datum, this is why we don't make any reference of it in the previous proposition.
\end{rema}

\section{The Hodge stratification}

As explained before, we have fixed $\pi \in \mathcal P$, and we assume that $\pi$ does not fall in case (AR). We have also defined Hodge polygons, and the generalized Rapoport locus is by definition the locus where this polygon is minimal. \\
We will now prove that this locus is dense.

\subsection{Lifting a module with filtration}

This is an intermediary section which contains some results concerning the existence of lifts of modules satisfying certain properties.

\begin{lemm} \label{lift_L}
Let $M$ be a free $k[[X]]$-module of rank $h$, and $N_1 \subset \dots \subset N_r \subset M$ be direct factors with $N_i$ of rank $d_i$. Let $\overline{L}$ be a $k$-vector subspace of $M \otimes_{k[[x]]} k$ of dimension $l$. There exists a lift $L$ of $\overline{L}$ such that in generic fiber the dimension of $L \cap N_i$ is $\max(0, l+d_i -h )$.
\end{lemm}

\begin{proof}
We prove this result by induction on the integer $r$. Let us consider the case $r=1$. \\
Define $\overline{M} = M \otimes_{k[[X]]} k$, and let $s$ be the dimension of $\overline{N_1} \cap \overline{L}$. Then there exists a a basis $e_1, \dots, e_{d_1}$ of $N_1$ such that the reduction of $e_1, \dots, e_s$ is a basis for $\overline{N_1} \cap \overline{L}$. One can then complete in a basis $e_1, \dots, e_h$ of $M$ such that the reduction of $e_1, \dots, e_s, e_{d_1+1}, \dots, e_{d_1+l-s}$ form a basis for $\overline{L}$. \\
If $l+d_1 \leq h$, one defines $L$ to be generated by $e_1 + X e_{d_1+l-s +1}, \dots, e_s + X e_{d_1+l}, e_{d_1+1}, \dots, e_{d_1+l-s}$. \\
If $l+d_1 >h$, one defined $L$ to be generated by \[e_1 + X e_{d_1+l-s +1}, \dots, e_{h+s-d_1-l} + X e_{h}, e_{h+s-d_1-l+1}, \dots, e_s, e_{d_1+1}, \dots, e_{d_1+l-s}.\]
Now let us turn to the general case. Let $\overline{L_0}$ be a complementary subspace of $\overline{L} \cap \overline{N_k}$ inside $\overline{L}$. One will lift the direct sum $\overline{L_0} \oplus \overline{L} \cap \overline{N_k}$. One will take an arbitrary lift of $\overline{L_0}$; by doing so one reduces to the case where $\overline{L} \subset \overline{N_k}$. Let $s$ be the dimension of $\overline{L_1} := \overline{L} \cap \overline{N_1}$. One will then distinguish two cases. \\
If $s \leq h - d_k$, one defines a lift $L_1$ of $\overline{L_1} $ such that $L_1 \cap N_k = \{ 0 \}$ in generic fiber. Considering a complementary subspace $\overline{L}'$ of $\overline{L_1} $ in $\overline{L}$, one can use the induction hypothesis by considering the modules $N_2 \subset \dots \subset N_k \subset M$. \\
If $s > h - d_k$, let $e_1, \dots, e_h$ be a basis of $M$ adapted to the filtration $N_1 \subset \dots \subset N_k \subset M$. Assume also that the reduction of $e_1, \dots, e_s$ is a basis for $\overline{L_1} $. Let $\overline{L_0} $ be the vector subspace of $\overline{L_1}$ spanned by the reduction of $e_1, \dots, e_{h-d_k}$. Let us consider the lift $L_0$ of $\overline{L_0} $ spanned by $e_1 + X e_{d_k+1}, \dots, e_{h-d_k} + X e_h$. Let $\overline{L'} $ be a complementary subspace of $\overline{L_0} $ in $\overline{L}$.
To lift $\overline{L'}$, we are thus reduced to lift it in $M' = \Vect(e_{s+1},\dots,e_{d_k})$, endowed with the filtration $N_1 \cap M' \subset N_{k-1} \cap M' \subset M'$.
We are thus reduced to the case of a smaller $k$, and by induction we are done.

\end{proof}

In the polarized case, one will use the following lemma.

\begin{lemm} \label{lift_C}
Let $M$ be a free $k[[X]]$-module of rank $2g$ with a perfect pairing, and $N \subset M$ be a totally isotropic direct factor rank $g$. Let $\overline{L}$ be a totally isotropic $k$-vector subspace of $M \otimes_{k[[x]]} k$ of dimension $g$. There exists a lift $L$ of $\overline{L}$ such that in generic fiber the dimension of $L \cap N$ is $0$.
\end{lemm}

\begin{proof}
Let us first consider the case where $\overline{L} = N \otimes_{k[[x]]} k$. Let $e_1, \dots, e_g$ be a basis of $N$, completed in a basis $e_1, \dots, e_{2g}$ of $M$, such that the pairing $\langle e_i , e_j \rangle$ is $1$ if $j = g+i$ and $0$ otherwise. One will define the module $L$ to be generated by the columns of the matrix $\left( \begin{array}{c}
I_g \\
XA 
\end{array} \right) $, where $I_g$ is the identity matrix and $A$ is any invertible symmetric matrix of size $g$. \\
Let us now turn to the general case. Define $\overline{M} := M \otimes_{k[[x]]} k$, $\overline{N} := N \otimes_{k[[x]]} k$ and $\overline{L_0} = \overline{L} \cap \overline{N}$. Let $\overline{L_1}$ be a complementary subspace of $\overline{L_0}$ in $\overline{L}$. Let $L_1$ be a totally isotropic lifting of $\overline{L_1}$ in $M$. One will look for a lift $L$ inside $L_1^\bot $ and containing $L_1$. One is then led to consider the module $M_1 := L_1^\bot / L_1$, which is free of rank $2(g-s)$, where $s$ is the dimension of $\overline{L_1}$. The module $N \cap L_1^\bot$ is free of rank $g-s$, and so is its projection onto $M_1$. By doing so, one is reduced to the previous case.
\end{proof}

\subsection{Density of the generalized Rapoport locus}

In this section, we use the previous lemmas to prove :

\begin{theor}
\label{thrRap}
The generalized Rapoport locus is (open and) dense.
\end{theor}

To prove this theorem, we do it for one $\pi$ at a time and find a lift "locally", i.e. for $G_\pi$. We will thus consider the possible cases : (AL) or (AU) and (C).

\begin{rema}
Again, the analogous result in case (AR) is false, as shown in the exemples in Appendix \ref{AppE}.
\end{rema}


\subsubsection{Cases (AL) and (AU)}
\label{sect321}

\begin{proof}
Let $x$ be a point of $X_\kappa := X\otimes \kappa$ and $G = G_\pi$ the associated $p$-divisible group. We want to prove that there exists a deformation of $x$ which lies in the generalized Rapoport locus. Thanks to Grothendieck-Messing, we will deform the Hodge filtration of $G$. For each $\tau \in \mathcal{T}$, one will use the previous lemma to deform $\omega_\tau$. By duality in case (AU), one automatically has a deformation of $\omega_{\overline{\tau}}$, hence of the whole of $\omega_G$. \\
Let us now describe the way to lift $\omega_\tau$. Let $D=H^1_{dR, \tau}$ ; it is a free as a $k[X]/(X^e)$ module. Let $M$ be the ($\tau$-part of the) evaluation of the crystal at 
$k[[t]]$; it is free as a $k[[t]][X]/(X^e)$, and reduces to $D$ modulo $t$. \todo{Notations à harmoniser. Donner un peu plus de détails} To lift $\omega_\tau$, one will successively lift $\omega_\tau^{[1]}, \dots, \omega_\tau^{[e]}$. First, one will 
lift $\omega_\tau^{[1]}$ in $M[X]$, the $X$-torsion of $M$. Let $L_1$ be any such lift. Then in order to lift $\omega_\tau^{[2]}$, one will work in $M_1 := X^{-1} L_1 / L_1$. One has 
the submodule $N_1 = M[X]/ L_1$. One will use the lemma \ref{lift_L} to lift $\omega_\tau^{[2]}$ to a module $L_2$ in such a way that the dimension of $L_2 \cap M[X]$ in generic 
fiber is $\max(d_1,d_2)$. Then one will consider $M_2 := X^{-1} L_2 / L_2$. One has the submodules $N_1' = (M[X] + L_2)/L_2$ and $N_2' = (M[X^2]\cap X^{-1}L_2 + L_2) / L_2$. Again, one will use the lemma \ref{lift_L}, and get a lift $L_3$ of $\omega_\tau^{[3]}$ such that in generic fiber the dimensions of $L_3 \cap M[X]$ and $L_3 \cap M[X^2]$ are respectively $\max(d_1,d_2, d_3)$ and $\max(d_1+d_2, d_1 + d_3, d_2+d_3)$. \\

\begin{lemm}
\label{lemma35}
There exists a lift $(\widetilde{\omega_\tau^{[i]}})$ of $\omega_\tau^{[1]} \subset \dots \subset \omega_{\tau}^{[e]} \subset D$ in $D\otimes k[[t]]$ satisfying the Pappas-Rapoport conditions and such that
\[ \dim (\widetilde{\omega_\tau^{[i]}})[X^j] = \max_{
0 <k_1 < \dots < k_j \leq i} d_{k_1} + \dots + d_{k_j}.\]
\end{lemm}

\begin{proof}
Indeed by induction we have proven the result for $i = 1,2,3$. Suppose it is true for $i \geq 1$, and denote
\[ M_i = X^{-1}L_i/L_i, \quad N_j = (M[X^j]\cap X^{-1}L_i + L_i)/L_i \quad \forall j \leq i, \quad \overline L = \omega_\tau^{[i+1]}/\omega_\tau^{[i]}.\]
A direct calculation shows that $\dim_{k[[t]]} M_i = h$, $\dim_k \overline L = d_{i+1}$, and $\dim_{k[[t]]} N_j = h - \dim (L_i[X^j]\backslash L_i[X^{j-1}])$.
We use the lemma \ref{lift_L} with these data to find $L$ such that
\[ \dim_{k[[t]]} L\cap N^j = \max(0,d_{i+1} - \dim (L_i[X^j]\backslash L_i[X^{j-1}]).\]
Let $\widetilde{\omega_\tau^{[i]}}$ the preimage of $L$ via $X^{-1}L_i \fleche M_i$, we thus have
\begin{eqnarray*} \dim \widetilde{\omega_\tau^{[i]}}[X^j] 
= \dim L_i\cap N^j + \dim L_i[X^j] 
= \max(L_i[X^j], d_{i+1} + \dim L_i[X^{j-1}]) 
\\
= \max_{
0 <k_1 < \dots < k_j \leq i+1} d_{k_1} + \dots + d_{k_j}.\end{eqnarray*}
Thus the lemma is proved.
\end{proof} 

\begin{lemm}
\label{lemma36}
Let $S$ be a $\kappa$-scheme and $\omega^{[e]} \subset (\mathcal O_S \otimes_{\mathcal O_{F^{ur}}} \mathcal O_F)^h$ be a sub-$\mathcal O_S \otimes \mathcal O_F$-module.
Then $\omega^{[e]}$ is in the generalized Rapoport locus for the datum $(d_i)_{1 \leq i \leq e}$ (i.e. $\Hdg(\omega^e) = \PR(d_i)$) if and only if for all $j$, \[ \dim \omega^{[e]}[X^j] = \max_{
0 <k_1 < \dots < k_j \leq e} d_{k_1} + \dots + d_{k_j}.\]
\end{lemm}

\begin{proof}
As the two properties are independent of the ordering of the values $(d_i)$ suppose to simplify $d_1 \geq d_2 \geq \dots \geq d_e$.
Then the last proposition means $\dim \omega^{[e]}[X^j] = d_1 + \dots + d_j$ which means exactly $\Hdg(i) = \PR(i)$.
\end{proof}

Reducing inductively  the datum given in lemma \ref{lemma35} to $k[[t]]/(t^n)$, one get a map by Grothendieck-Messing
(note that $k[[t]]/(t^n) \fleche k[[t]](t^{n-1})$ is endowed with nilpotent divided powers),
\[ \Spf(k[[t]]) \fleche X_\kappa.\]
But as $X_\kappa$ is a scheme, this induces a map $\widetilde{x} : \Spec(k[[t]]) \fleche X_\kappa$, generizing our point $x$, and such that in generic fiber the module 
$\widetilde{x}^*\omega_\tau[1/t]$ is given by $\widetilde{\omega_\tau}[1/t]$, thus lies in the generalised Rapoport locus by Lemma \ref{lemma36}.
\end{proof}

\subsubsection{Case (C) }

\begin{proof}
One will keep the same notations as the previous section. One will lift the module $\omega_\tau$ in $M$. But in this case, the module $M$ has a perfect pairing, and one needs to consider totally isotropic lifts. \\
One starts by considering the module $M_1 := M[X]$. This module has a perfect pairing $h_1$ (induced by the one on $M$, see section \ref{sect233} and remark \ref{isot}), and we can lift the module $\omega_\tau^{[1]}$ to a module $L_1 \subset M_1$, which will still be totally isotropic for $h_1$. Then one consider the module $M_2 := X^{-1} L_1 / L_1$. Since $L_1$ is totally isotropic in $M_1$ for $h_1$, the pairing $h_2$ induces a pairing on $M_2$. Using the lemma \ref{lift_C}, one will take a lift $L_2$ of $\omega_\tau^{[2]}$, such that $L_2 / L_1$ is totally isotropic in $M_2$ for $h_2$, and disjoint from $M[X]/L_1$ in generic fiber. \\
One repeat this process and gets lifts $L_1 \subset L_2 \subset \dots \subset L_e$. In generic fiber, the multiplication by $X$ is an isomorphism between $L_{i+1} / L_i$ and $L_i / L_{i-1}$ for every $1 \leq i \leq e-1$, and the lift $L_e$ of $\omega_\tau$ is thus generically free as a $\mathcal O_F \otimes_{\tau} k[[t]] = k[[t]][X]/(X^e)$-module. 
As before, by Grothendieck-Messing this leads to an algebraisable map
\[ \Spf(k[[t]]) \fleche X_\kappa,\]
generizing $x$, and whose generic fiber lies in the Rapoport locus.
\end{proof}

\subsection{Futher strata}

As the Hodge stratification is constructed using the Nilpotent cone of some $\GL_n$, for which the stratification is a strong stratification, we can investigate the same question for $X$. First recall the definition of a strong stratification.

\begin{defin}
Let $X$ be a topological space. A (\textit{weak}) stratification of $X$, with respect to a partially ordered set $(I, \leq)$, 
 is a decomposition
 \[ X = \coprod_{i \in I} X_i,\]
 such that $\overline{X_i} \subset \coprod_{j \leq i} X_j$. A (\textit{weak}) stratification is a \textit{strong} stratification if moreover
 \[ \overline{X_i} = \coprod_{j \leq i} X_j.\]
\end{defin}

\begin{example}
\begin{enumerate}
\item In the case of a unramified PEL datum, the Hodge stratification of $X_\kappa$ is a \textit{strong} stratification (this is trivial as there is only one stratum).
\item Still in the case of an unramified PEL datum, the Newton stratification and the Ekedahl-Oort stratification of $X_\kappa$ are strong stratifications (see  \cite{WedVieh} Theorem 2 and \cite{Ham} Theorem 1.1).
\end{enumerate}
\end{example}

\begin{prop}
In general, the Hodge stratification is not a strong stratification.
\end{prop}

\begin{proof}
This has nothing to with do with abelian varieties but rather with the space of partial flag on a fixed space together with a nilpotent operator (called \textit{Spaltelstein Varieties}).
Let $K$ be any field, $V = K^6$ and suppose given a full flag
\[ 0 \subset V_1 \subset V_2 \subset V_3 = V,\]
with $\dim_K V_i = 2i$. Suppose moreover that there is $\pi \in \End_K(V)$ such that $\pi(V_i) \subset V_{i-1}$ for all $i \geq 1$. This corresponds to (the local model of) a Pappas-Rapoport datum (AL) with $e = 3$ and $d_1 = d_2 = d_3 = 2$.
Fix a basis $e_1,\dots,e_6$ of $V$ such that $e_1,e_2$ is a basis of $V_1$, $e_1,\dots,e_4$ a basis of $V_2$. Then set the following choices for $\pi$ in this basis
\[
\pi_1 = \left(
\begin{array}{ccc}
0  & I_2  &  0 \\
 0 &  0 &  0 \\
 0 &  0 & 0  
\end{array}
\right) \quad \text{and} \quad 
\pi_2 = \left(
\begin{array}{cc}
0  & 
\left(
\begin{array}{ccc}
 1 &0 & 0 \\
0 & 1 &  0\\
 0 & 0 &   1
\end{array}
\right) \\
0 & 0
\end{array}
\right).
\]
In both cases $\pi$ satisfies the condition of a Pappas-Rapoport datum, with $\pi_j^2 = 0$. The Hodge polygons are associated with the partitions $(4,2)$ and $(3,3)$ of $6$, and $(3,3) \leq (4,2)$, but there exists no deformation from $((V_i),\pi_1)$ to a space with Hodge polygons $(3,3)$ which satisfies the Pappas-Rapoport condition.
Indeed, if it were the case then the Hodge polygon of $(\pi_1)_{|V_2}$ will descend by generisation, thus would remain the same and thus the $\pi$-torsion of the deformation should intersect $V_2$ only along $V_1$, and $\pi$ will send $V_2$ surjectively to $V_1$. Now that means, as $\pi$ sends $V_3$ to $V_2$ and $\pi^2$ send $V_3$ to $0$ (as it is of Hodge polygon $(3,3)$) thus $\pi$ sends $V$ to $V_1$, and thus the kernel of $\pi$ is of rank 4, a contradiction.

Denote $X$ the moduli space of all possible $((V_i),\pi)$ with $d_i = 2$ and $\pi(V_i) \subset V_{i-1}$. Denote $X_{(3,3)}$ the Hodge stratum corresponding to the Hodge polygon of partition $(3,3)$ and $X_{(4,2)}$ the analogous one. If $\overline{X_{(3,3)}} \supset X_{(4,2)} \ni x = ((V_i),\pi_1)$ take $C$ an irreducible component of $X_{(3,3)}$ passing through $x$, and look at the local ring of $C$ at $x$, $\mathcal O_{C,x}$. This induces a generisation of $x$ such that the Hodge polygon is above $(3,3)$, under $(4,2)$ and by the previous calculation can't be equal to $(3,3)$. Thus it is generically $(4,2)$ too. This is true for all components $C$, thus, locally at $x$, $X_{(4,2)}$ is an open component of $X_{(3,3)}$, and thus 
$x \not\in \overline{X_{(3,3)}}$.
\end{proof}

We still hope to construct a strong stratification on $X_\kappa$, by "cutting" in parts the Hodge stratums. Unfortunately, the situation gets very complicated when the ramification index $e$ grows. One has however the following result when $e=2$.
Recall that in polarised cases ((C),(AU) and (AR)) we only consider symmetric data $(\nu_\tau)_\tau$ for the Hodge strata.  
\begin{prop}
\label{prop38}
If $e \leq 2$ 
and no $\pi$ falls in case (AR), then the Hodge stratification is a strong stratification.
\end{prop}

\begin{proof} If $e =1$, there is only one Hodge stratum and everything is trivial. 
If $e=2$, we can suppose $d_{\tau,1} \geq d_{\tau,2}$. Indeed in case $C$ there is an equality, and in case AL or AU considering the dual group (which coincide with $G^{(s)}$ in case (AU), thus we can in case (AU) consider only half of the embeddings $\tau$ as we did in the proof of theorem \ref{thr229}) we can reduce to this case.
The $\tau$-Hodge polygon is given by two integers $a_{\tau,1} \geq a_{\tau,2}$ such that if $x \in X_\kappa(k)$, then $\omega_{\tau,x}[\pi]$ is a $k$-vector space of dimension
 $a_{\tau,1}$ and $\omega_{\tau,x}$ is of dimension $a_{\tau,1}+a_{\tau,2}$. As $\omega_{\tau,x}^{[1]}$ is of dimension $d_{\tau,1}$ and of $\pi$-torsion by the 
 Pappas-Rapoport condition, we have that $d_{\tau,1} \leq a_{\tau,1}$. Thus, all the possible $\tau$-Hodge polygons are classified by couples $(a_{\tau,1},a_{\tau,2})$ with 
 $a_{\tau,1} \geq d_{\tau,1}$ and $a_{\tau,1} + a_{\tau,2} = d_{\tau,1} + d_{\tau,2}$. Moreover if $(a_{\tau,1},a_{\tau,2})$ and $(b_{\tau,1},b_{\tau,2})$ are two $\tau$-Hodge 
 polygons, the former is above the latter if and only if $a_{\tau,1} \geq b_{\tau,1}$. The Generalised Rapoport locus corresponds to 
 $(a_{\tau,1},a_{\tau,2}) = (d_{\tau,1},d_{\tau,2})$ Thus we will prove that given any point $x \in X(k)$ with $\tau$-Hodge polygon $(a_{\tau,1},a_{\tau,2})$ 
 not in the generalised Rapoport locus, there is a deformation to $k[[t]]$ with $\tau$-Hodge polygon $(a_{\tau,1}-1,a_{\tau,2}+1)$. Fix such a point, and fix a 
 $k[\pi]/\pi^2$ basis of $H^1_{dR,\tau}$, $e_1,\dots,e_h$, such that $\pi e_1,\dots, \pi e_{d_1}$ is a basis of $\omega_\tau^{[1]}$ over $k$ and 
 $e_1,\dots,e_r,\pi e_{d_1+1},\dots,\pi e_{d_1+s}$ (necessarily with $r \leq d_1$) induces a basis of $\omega_{\tau}/\omega_\tau^{[1]}$. 
Then $r +s = d_2$ and $a_{\tau,1} = d_1 + s$. As this point is not in the generalised Rapoport locus, we have $s > 0$, thus $r < d_1$. Then set in $H^1_{dR,\tau} \otimes_{k[\pi]/\pi^2} (k[\pi]/\pi^2)[[t]]$, 
\[ \widetilde{\omega_\tau}^{[1]} = k[[t]](\pi e_1,\dots,\pi e_{d_1}) \quad \text{and} \quad \widetilde{\omega_\tau} = \widetilde{\omega_\tau}^{[1]} + k[[t]](e_1,\dots,e_r,e_{d_1+s} + te_{r+1},\pi e_{d_1+1},\dots,\pi e_{d_1+s-1}).\] In case C, we need to choose a lift of $\Vect(e_1,\dots,e_{r},\pi e_{d_1+1},\dots,\pi e_{d_1+s})$ in $\pi^{-1}\widetilde{\omega_\tau}^{[1]}/\widetilde{\omega_\tau}^{[1]}$ (which has a perfect pairing), whose intersection with $H^1_{dR,\tau}[\pi]/\widetilde{\omega_\tau}^{[1]}$ is of dimension $s-1$. But we can quotient further by $\Vect(e_1,\dots,e_r)$ (which is totally isotropic), and we are reduced to the case $r =0$. In this case this is as in the proof of the first part of Lemma \ref{lift_C}, taking the matrix $A$ to be symmetric of rank 1.

Then, as in the proof of subsection \ref{sect321} there is a lift of $x$ whose Hodge filtration is given by $\widetilde{\omega_\tau}$. Moreover,
\[ \dim_{k((t))^{perf}} (\widetilde{\omega_\tau}\otimes k((t))^{perf})[\pi] = d_1+s-1.\]
\end{proof}

\begin{rema}
The calculation of Appendix \ref{AppE} still shows that even when $e =2$, the analogous result in case AR is false.
\end{rema}

\section{The $\mu$-ordinary locus}

\subsection{Density of the $\mu$-ordinary locus}

The goal of this section is to show the following theorem,
\begin{theor} 
\label{thr41}
Let $\pi$ be a prime as in section \ref{sectdecpol}, and suppose that $\pi$ doesn't fall in case (AR).
Then the $\mu$-ordinary locus (for $\pi$) $X^{\nu = PR}$ inside $X$ is dense.
\end{theor}

To prove it, we will once again follow the strategy of deformations of the $p$-divisible group (by Serre-Tate's theorem), do it one prime at a time and thus in each of the cases (AL),(AU),(C). Moreover, by \cref{thrRap}, we only need to deform $p$-divisible group that are already in the generalised Rapoport locus. Our main tool is \cref{Wed328}. From now on we will only consider lifts of a crystal in the sense of section 3.2.3 of \cite{Wed1}. We thus call a \textit{deformation} of a crystal over $k$ a display over $k[[t]]$ of the form $P_\alpha$ for some $\alpha \in \Hom_{W(k[[t]])}(P, W(tk[[t]])P)$ in the notations of \cite{Wed1}.

\begin{rema}
In case (AR) our argument breaks down. It is likely that in general the $\mu$-ordinary locus is not dense in every irreducible component, similarly to the Hilbert case at Iwahori level, described in \cite{Stamm}. This is confirmed by the case of $U(1,n-1)$, calculated by \cite{Kramer}, whose calculation is made in Appendix \ref{AppE}.
\end{rema}

Moreover, for cases (AU) and (C), we will have to use a slightly more adapted polarisations. Indeed, the object considered have a natural polarisation, compatible in a certain sense to the additional action of a ring $\mathcal O$, but not $\mathcal O$-hermitian (or bilinear). Denote $L/L^+$ the extension of (local) fields at the prime considered (thus $L = L^+$ in case C), and denote $s \in \Gal(L/L^+)$ the non trivial automorphism (if it exits, $s=\id$ otherwise). Denote $\Diff^{-1} = \Diff^{-1}_L$ the inverse different of $L$, and $\mathcal O = \mathcal O_L$.

We have the following,

\begin{prop}
\label{proppolLan}
Let $k$ be a perfect field of characteristic $p$, and $G$ be a $p$-divisible $\mathcal D$-module. Its Dieudonne module $M$, a free $W_{}(k)\otimes_{\ZZ_p}\mathcal O$-module, together with two applications
\[ V : M \fleche M, \quad \text{and} \quad F : M \fleche M,\]
which are $\sigma^{-1}$-(resp $\sigma$-)linear, satisfying $FV = VF = p\id$, is endowed with a $s$-antihermitian $W_{\mathcal O}(k)$-pairing,
\[ h : M \times M \fleche W(k)\otimes \Diff^{-1},\]
satisfying
\[ h(x,Fy) = h(Vx,y)^{\sigma},\quad \forall x,y \in M.\]
Moreover, if $<.,.>$ denote the original alternating pairing on $M$, we have $\tr_{F/\QQ_p} h = < . ,.>$. Such a $h$ is unique.
In particular in case (C), $h$ is alternating.
\end{prop}

\begin{proof}
The existence of $h$ $s$-antihermitian satisfying $tr h = <, >$ is \cite{Lan} Lemma 1.1.4.5. Thus, it suffices to prove the compatibility with $F$ and $V$.
But $ \forall x,y \in M, o \in \mathcal O$,
\begin{eqnarray*} \tr o h(x,Fy) = \tr h(ox,Fy) =  <ox,Fy> = <Vox,y>^\sigma = <\sigma^{-1}(o)Vx,y>^\sigma \\= \tr(h(\sigma^{-1}(o)Vx,y)^\sigma) = \tr(o h(Vx,y)^\sigma).\end{eqnarray*}
Thus, by \cite{Lan} Corollary 1.1.4.2, we have $h(x,Fy) = h(Vx,y)^{\sigma}$, for all $x,y \in M$.
\end{proof}

Unfortunately, it will not be possible to always suppose $h$ alternating, even if $< , > = \tr h$ is, as shown by the following example :

\begin{example}
Let $\CC  = \RR^2$ endowed with the $\RR$-linear alternating pairing $\CC \times \CC \fleche \RR$ given in basis $(1,i)$ of $\CC$ by 
\[
\left(
\begin{array}{cc}
  0 & -1     \\
  1 &      0
\end{array}
\right),
\]
whose $h$ is given by $\frac{i}{2} z_1 \overline{ z_2}$.
\end{example}

But in case (C), as $s \in \Gal(L/L^+)$ satisfies $s = \id$, we have that $h$ is anti-symmetric, thus alternating (as $\Char(W_{\mathcal O}(k)) \neq 2$). Remark that the same proposition is true for displays.


\subsection{The case $(AL)$} 
\label{thrAL} 

Let $x_G$  a point in the Rapoport locus, with values in a perfect field $k$, corresponding to a group $G$ and fix $\tau_0$. Note that any generisation of $x$ still lives in the Rapoport locus. 
Let $M$ the Dieudonné crystal of $G$ over $k$.
\begin{lemm}
\label{lemslope0}
We can suppose that for all $\tau$, the first slope of the Hodge-Polygon is zero. In particular, for all $\tau$, there exists $x_\tau \in M_\tau$ such that
\[ F_\tau(x_\tau) \not\equiv 0 \pmod{\pi}.\]
\end{lemm}

\begin{proof}
Indeed, otherwise denote by $a_\tau$ the first slope for all $\tau$, this means that $F_\tau$ is divisible by $\pi^{a_\tau}$ on the Dieudonné module of $G$.
Denote $F^0_\tau = \frac{1}{\pi^{a_\tau}}F_\tau$ and $V^0_\tau = \pi^{a_\tau}V_\tau$. Denote by $G'$ the $p$-divisible group associated to $(M,F^0,V^0)$. Then it is easily checked that the association $G \mapsto G'$ is bijective on $p$-divisible group with $\mathcal O$-action with fixed $\tau$-Hodge polygons on the source to fixed $\tau$-Hodge polygons  where each $\tau$-slope is decreased by $a_\tau$. Moreover this is compatible with display-deformations to $k[[t]]$ and specialisation to $k((t))^{perf}$ in \cref{Wed328}.
Thus we only need to deform $G'$, whose first Hodge-slope is zero for all $\tau$.
\end{proof}

Thus, the first slope of $\Hdg_\tau(G)$ is zero for all $\tau$. If the first slope for $Newt(x)$ is also zero, then there is a splitting
\[ G = G^0 \times G',\]
where the Hodge polygon and Newton polygon of $G'$ does not have a slope zero (see \cite{BH}, Théorème 1.3.2). If we can find a deformation of $G'$ which is $\mu$-ordinary, then we are finished.
Thus up to exchange $G$ by $G'$, we can suppose that the first slope of $Newt(G)$ is non zero and proceed by induction of the height of $G$.

Let $x \in M_{\tau_0}$ such that $F(x) \not\equiv 0 \pmod{\pi}$ (this is possible by the previous lemma). Let $i$ be the minimal integer such that 
\[ F^{i}(x) \equiv 0 \pmod \pi.\]
As the Newton polygon of $G$ doesn't have a zero slope, there exists such a $i$. By what preceed, we know that $i \geq 2$.

\begin{lemm}
There exists a generisation of $x_G$ such that $i > f$. 
\end{lemm}

\begin{proof}
Suppose $i \leq f$. Thus $F^i(x) \equiv 0 \pmod \pi$. By the previous lemma; there is $x_{i-1} \in M_{\sigma^{i-1}\tau}$ such that $F(x_{i-1}) \not\equiv 0 \pmod\pi$. 
As $y = F^{i-1}(x) \not\equiv 0 \pmod \pi$ but $F(y) \equiv 0 \pmod\pi$, $(x_{i-1},y)$ is a linearly independant family in $M_{\sigma^{i-1}\tau}/\pi M_{\sigma^{i-1}\tau}$. Define an homomorphism of $M/\pi M$ by
\[ N_{\tau'} = 0, \quad \forall \tau' \neq \sigma^{i-1}\tau \quad \text{and} \quad N_{\sigma^{i-1}\tau}y = x_{i-1},  \quad N_{\sigma^{i-1}\tau}x_{i-1} = 0,\]
and extend by zero $N$ on a complementary basis of $M_{\sigma^{i-1}\tau}$. Denote by $N$ any nilpotent lift of $N$ to $M$.
Define $D_N$ the extension of $M$ to $W(k[[u]])$ given by $N$ as in \cite{Wed1} (see also \cite{Zink}).
We can calculate
\begin{eqnarray*} F_N^i(x) = F_N(F_N(F_N^{i-2}(x))) = F_N(F_N(F^{i-2}(x))) = F_N(y\otimes 1 + x_{i-1} \otimes u) \equiv F_N(x_{i-1}\otimes u) \\
\equiv F(x_{i-1})\otimes u \not\equiv 0 \pmod \pi.\end{eqnarray*}
Thus over $W(k((u))^{perf})$ the display $D_N$ correspond to a $p$-divisible group such that $i' > i$ by Proposition \ref{Wed328}. By induction we get the result.
\end{proof}

\begin{lemm}
The exists a generisation of $x_G$ such that the generic fiber is not infinitesimal.
\end{lemm}

\begin{proof}
By the previous lemma, we can suppose $F^f(x) \not\equiv 0 \pmod\pi$. If $G$ is not itself infinitesimal,
let $r_0$ be the minimal integer such that
\[ F^{r_0f}(x) \equiv 0 \pmod \pi.\]
The family $(x,F^f(x),\dots,F^{(r_0-1)f}(x))$ is linearly independent mod $\pi$.
Indeed, suppose we are given
\[ \lambda_0x + \lambda_1 F^f(x) + \dots + \lambda_{r_0-1}F^{(r_0-1)f}(x) \equiv 0 \pmod \pi;\]
and denote $i$ the smallest integer such that $\lambda_i \neq 0 \pmod \pi$. Then,
\[ F^{if}(x) = \lambda_i^{-1}(\lambda_{i+1} F^{(i+1)f}(x) + \dots + \lambda_{r_0-1}F^{(r_0-1)f}(x)),\]
and thus $F^{(r_0-i-1)f}(F^{if}(x)) = F^{(r_0-1)f}(x) \equiv 0 \pmod \pi$, which is impossible.
Set $N$ such that 
\[ N_{\tau'} = 0, \forall \tau' \neq \tau, \quad Nx = NF^{if}(x) = 0, \forall i \neq 1 \quad\text{and}\quad NF^f(x) = x.\]
Set $D_N,F_N$ the associated display over $W(k[[u]])$ which reduces to $(M,F)$. We calculate,
\[ F_N^f(x) = F^f(x) + uNF^f(x) = F^f(x) + ux,\]
and more generally,
\[ F_N^{if}(x) = F^{if}(x) + uF^{(i-1)f}(x_0) + \dots + u^ix_0.\]
In particular,
\[ F_N^{r_0f}(x) = 0 + uF^{(r_0-1)f}(x) + \dots + u^{r_0}x \neq 0 \pmod \pi.\]
Thus, the base change to $W(k((u))^{perf})$ of $D_N$ satisfies $r_0(D_N) > r_0(M)$. By induction we can suppose that $r_0 > \dim D$, thus that $D$ is not infinitesimal.
\end{proof}

By induction on the number of Newton and Hodge slopes of $G$ that are not equal (see \cref{lemslope0}), we get a chain of generisations starting to $x_G$ and ending to a $\mu$-ordinary point.

\begin{cor}
In case (AL), the $\mu$-ordinary locus $X_{\mu-ord}^{PR}$ is dense.
\end{cor}

\subsection{The case $(AU)$}

In the linear case, as our deformation by $N$ will be polarised, which means that we deform in the same time $F_\tau$ and $F_{\overline\tau}$, these two deformations might cancel out. Thus, we will work almost as in \cite{Wed1}, by finding a deformation sequence (which assures that there will be no cancelations when calculating $F^{2f}_N$).
We denote $\overline M = M\pmod \pi$.
\begin{defin}
$(x_\tau) \in \overline M = \bigoplus_\tau \overline M_\tau$ is a deformation sequence if
\begin{enumerate}
\item $x_\tau \in \overline M_\tau$ for all $\tau$,
\item $Fx_\tau \not\equiv 0 \pmod \pi$ for all $\tau$,
\item If $F^2x_\tau \not\equiv 0 \pmod \pi$, then $Fx_\tau = x_{\sigma\tau}$.
\end{enumerate}
\end{defin}

We will use \cref{Wed328} to deform a given $p$-divisible $\mathcal O$-module to a $\mu$-ordinary one. By \cref{thrRap} we can suppose that the $p$-divisible group $G$ associated to a point $x_G \in X^{PR}$ we start with is in the Rapoport locus, i.e. $\Hdg(x_G) = \PR(x_G)$.
In particular, the first slope of $\Hdg(x_G)$ coincide with the one of $\PR(x_G)$. We will construct the deformation by induction, using that if $G$ is not bi-infinitesimal, we can decompose,
\[ G = G^{et} \times G^{bi} \times G^m,\]
where $G^{et}$ is etale and $G^{m}$ is multiplicative, and $G^{bi}$ is bi-infinitesimal. If $G^{bi}$ is $\mu$-ordinary, then so is $G$, and thus we only need to deform $G^{bi}$.
From now on, suppose $G$ is bi-infinitesimal and we will prove that there exists a deformation of (some modification of) $G$ that is $\mu$-ordinary but not necessarily bi-infinitesimal generically. By induction on the Newton polygon of the bi-infinitesimal part, we will then deduce that we can deform $G$ to a $\mu$-ordinary $p$-divisible group. 
Just like in the beginning of section \ref{thrAL} we can suppose that the first slope of $\Hdg_\tau(x_G)$ is zero for all $\tau$,
 but here there will be some minor complications, and we need to introduce the notion of a non-necessarily parallel polarized $\mathcal O$-crystal. 

\begin{defin}
\label{defseq}
A non-necessarily parallel polarized $\mathcal O$-crystal (of type AU) over a perfect field $k$ is a tuple $(M,V,F,\iota,h(\cdot,\cdot))$ where 
\begin{enumerate}
\item $M$ is a free $W(k)$ module,
\item $F$ is $\sigma$-linear and $V$ is $\sigma^{-1}$-linear,
\item $\iota : \mathcal O \fleche \End_{W(k)}M$ such that $F(\iota(x)m) = \iota(\sigma(x))F(m)$ and $V(\iota(x)m) = \iota(\sigma^{-1}(x))V(m)$ for all $x \in \mathcal O, m \in M$.
\item The pairing 
\[ h( \cdot, \cdot ) : M \times M \fleche  W(k)\otimes \Diff^{-1},\]
is perfect, anti-hermitian, and satisfies
\[ \iota(x) h(m,n) = h( \iota(x)m,n) = h( m,\iota(\overline x)n), \quad \forall x \in \mathcal O,m,n \in M,\] 
\end{enumerate}
such that we can decompose
\[ M = \bigoplus_{\tau : F^{ur}\hookrightarrow \CC} M_\tau,\]
and thus accordingly we have maps,
\[ F_\tau : M_\tau \fleche M_{\sigma\tau} \quad \text{and} \quad V_\tau : M_{\sigma\tau} \fleche M_\tau,\]
such that there exists integers $(a_\tau)_{\tau \in \mathcal T}$ satisfying for all $\tau$, $V_\tau F_\tau  = p\pi^{-a_\tau}\overline{\pi}^{-a_{\overline\tau}} Id_{M_\tau}$, $F_\tau V_\tau = p\pi^{-a_\tau}\overline{\pi}^{-a_{\overline\tau}} Id_{M_{\sigma\tau}}$ and 
\[ h(F_\tau x_\tau, x_{\sigma\overline\tau}) = h(x_\tau,V_{\overline\tau}x_{\sigma\overline\tau})^{\sigma} \quad \text{and}\quad h(V_\tau x_{\sigma\tau}, x_{\overline\tau}) = h(x_{\sigma\tau},F_{\overline\tau}x_{\overline\tau})^{\sigma^{-1}}.\]
The tuples of integer $(a_\tau)_{\tau \in \mathcal T}$ is called the amplitude of $(M,F,V,\iota)$.
\end{defin}

\begin{example}
A polarised $\mathcal O$-crystal is a particular case of a non-necessarily parallel polarized $\mathcal O$-crystal for which the amplitude is constant equal to 0, i.e. $a_\tau =0$ for all $\tau$.
\end{example}

\begin{defin}
We can define the Hodge polygons of a non-necessarily parallel $\mathcal O$-crystal $(M,F,V,\iota)$ in the same way as a regular crystal, by defining the polygon
\[ \Hdg_\tau(M,F,V,\iota)(i) = \frac{c_1 + \dots + c_i}{e},\]
where $d = \dim_{\mathcal O} M_\tau$ and $M_\tau/FM_{\sigma^{-1}\tau} \simeq \bigoplus_{i=1}^d W_{\mathcal O}(k)/\pi^{c_i}W_{\mathcal O}(k)$, and $c_1 \leq c_2 \leq \dots \leq c_d$.
\end{defin}

\begin{prop}
\label{propNNP}
Suppose $k$ is perfect and contains $\kappa_F$.
To any polarised $\mathcal O$-crystal  $(M,F,V,\iota,h)$, we can associate a 
non-necessarily parallel 
polarised $\mathcal O$-crystal 
\[ (M^0,F^0,V^0,\iota^0,h^0),\]
for which the first slope of $\Hdg_\tau(M^0,F^0,V^0,\iota^0)$ is zero. 
Moreover, if $a_\tau$ denote for all $\tau$ the first slope of $\Hdg_\tau(M,F,V,\iota)$, then $(M^0,F^0,V^0,\iota^0,h^0)$ is of amplitude $(a_\tau)_\tau$. 
Moreover, we have $(M^0,\iota^0,h^0) = (M,\iota,h)$ and for all $\tau$,
\[ F_\tau = \pi^{a_\tau}F_\tau^0 \quad \text{and}\quad V_\tau = \overline{\pi}^{a_{\overline\tau}}V_\tau^0.\]
If $(P,Q,V^{-1},F,\iota,h)$ is the base change of $(M,F,V,\iota,h)$ to $k[[t]]$ (i.e. the polarized $\mathcal O$-deformation with $N = 0$ in Proposition \ref{Wed328}), we can also associate to it another (non necessarily parallel) display by the same operation, such that the Dieudonne modules of the base change over $k$ and $k((t))^{perf}$ coincides with the previous construction for crystals. In particular, the association $D \fleche \mathcal P \fleche D' = \mathcal P \otimes k((t))^{perf}$ does not change the Hodge polygons.
\end{prop}

\begin{proof}
Indeed, $M$ is a module over $W(k) \otimes_{\ZZ_p}\mathcal O$, which we can split $M = \bigoplus_{\tau} M_\tau$ over $\tau : \mathcal O^{ur} \hookrightarrow \mathcal O_C$.
As $F$ is $\sigma$-linear, $F_\tau : M_\tau \fleche M_{\sigma\tau}$, and $F_\tau$ is generically inversible, thus $M$ is free as $W(k)\otimes_{\ZZ_p}\mathcal O$-module.
By hypothesis, $F_\tau$ is divisible by $\pi^{a_\tau}$, thus set $F^0 = \pi^{-a_\tau} F_\tau$. As $F_{\overline\tau}^\vee = V_{\tau}$, we have that $V_{\tau}$ is also divisible by $\pi^{a_{\overline\tau}}$, thus by $\overline{\pi}^{a_{\overline\tau}}$ and we can set $V_\tau^0 = \overline{\pi}^{-a_{\overline\tau}}V_\tau : M_{\sigma\tau} \fleche M_{\tau}$. By proposition \ref{proppolLan}, $M$ is endowed with an hermitian pairing
\[ h : M \times M \fleche W(k) \otimes_{\ZZ_p} \Diff^{-1},\]
such that $\tr h = <.,.>$ and $h(x,Fy) = h(Vx,y)^\sigma$. Thus, for $x \in M_{\tau},y \in M_{\sigma\overline\tau}$,
\begin{eqnarray*}
h(\pi^{a_\tau}F_\tau^0x,y) =\overline{\pi}^{a_\tau}h(F_\tau^0x,y), \quad \text{and} \quad h(x,\overline{\pi}^{a_{\tau}}V_{\overline\tau}^0y) = h(\pi^{a_\tau}x,V_{\overline\tau}^0y).
\end{eqnarray*}
Thus,
\[ h(F_\tau^0x,y) = h(x,V_{\overline\tau}^0y)^\sigma.\]
$M$ is free over $W(k)\otimes \mathcal O$, and we have that $V^0 M_\tau = \frac{1}{\pi^{a_{\overline\tau}}} VM_{\tau}$ and thus as a display over $W(k)$
\[ (V^0)^{-1} : V^0 M_\tau \overset{\overline{\pi}^{a_{\overline\tau}}}{\fleche} VM_\tau \overset{V^{-1}}{\fleche} M_\tau.\]

Now if $(P,Q,V^{-1},F)$ is the base change of $(M,F,V,\iota,h)$, we have that $P$ is free over $W(k[[t]])\otimes_{\ZZ_p}\mathcal O$.
 We can thus decompose $P$ and $Q$ over $W(R) \otimes_{\ZZ_p}\mathcal O = \prod_\tau W(R)_\tau\otimes_{\mathcal O_\tau^{ur}}\mathcal O = \prod_\tau W(R)_\tau[\pi]/(E_\tau(\pi)).$

We have the $\sigma$-linear morphism
\[ V_\tau^{-1} : Q_{\sigma\tau} \fleche P_\tau,\]
and we set $Q^0 = VM^0 \otimes_{W(k)}W(k[[t]])$ with 
$(V_\tau^{0})^{-1} = (1\otimes\overline\pi^{a_{\overline\tau}})V_\tau^{-1} : Q^0_{\sigma\tau} \overset{\overline\pi^{a_{\overline\tau}}\otimes 1}{\fleche} Q_{\sigma\tau} \overset{V_\tau^{-1}}{\fleche} P_\tau$. As $P_\tau$ is a free $W(R)_\tau[\pi]/(E_\tau(\pi))$-module,
 thus we can also divide $F_{\tau,P} = F_{\tau,M} \otimes 1$ by $\pi^{a_\tau}$ and the rest is a simple verification. As $Q_{\sigma\tau} = \pi^{a_{\overline\tau}} Q^0_{\sigma\tau}$, we have that both 
the $\tau$-Hodge polygons of $D$ and $D'$ are those of $Q^0$, i.e. of $M^0$, with $\tau$-slopes increased by $a_\tau$.  
\end{proof}

\begin{prop}
Denote by $M$ a polarised $\mathcal O$-crystal, and $M^0$ the associated non-parallel crystal associated to it above.
Suppose we have a (polarised, $\mathcal O$-)deformation $\mathcal P^0$ of $M^0$ by a NNP-display over $k[[t]]$ given as in Proposition \ref{Wed328}, then we have an associated (polarised, $\mathcal O$-)deformation $\mathcal P$ of $M$ (with $\mathcal P$ a parallel display) such that $(\mathcal P \otimes_{k[[t]]} k((t))^{perf})^0 = \mathcal P^0 \otimes_{k[[t]]} k((t))^{perf}$.
If moreover the crystal $D^0 = \mathcal P^0 \otimes k((t))^{perf}$ is not bi-infinitesimal, then we can decompose,
\[ D^0 = D^{0,et} \times D^{0-bi} \times D^{0,mult},\]
and we have an associated decomposition of $D = \mathcal P \otimes k((t))^{perf}$,
\[ D = D^{a-et} \times D^{bi} \times D^{a-mult},\]
where $D^{a-et}$ is isocline of slope $\frac{1}{ef}\sum_\tau a_\tau$ and $D^{a-mult} = (D^{a-et})^\vee$. Moreover the association $M \mapsto D$ does not change the Hodge polygons.
\end{prop}

\begin{proof} Let $\mathcal P^0 = (P^0,Q^0,F^0,(V^0)^{-1})$ be the display associated to the base change of $M^0$ to $k[[t]]$, and $\mathcal P$ the analogous display for $M$.
Suppose given $N$ a $W(k)$-linear morphism of $P_0 \otimes_{k[[t]]} k$ of square zero, which is $\mathcal O$-linear, and skew-symmetric. We thus have a deformation 
$\mathcal P^0_N$ by setting $F^0_N = (\id + [t]N)F^0$ and $(V^0_N)^{-1} = (\id + [t]N) (V^0)^{-1}$. Let us set $\mathcal P_N$ the analogous deformation for $\mathcal P$.

We claim that the crystal $\mathcal P_N \otimes_{k[[t]]} k((t))^{perf}$ satisfies $(\mathcal P_N \otimes_{k[[t]]} k((t))^{perf})^0 = \mathcal P^0_N \otimes_{k[[t]]} k((t))^{perf}$.
Indeed, as $F_N = (\id + [t]N)F$ on $\mathcal P$, and $F$ is the pullback of $F$ on $\underline M$, we have that $F_\tau$ is divisible by $\pi^{a_\tau}$ for all $\tau$, thus on $\mathcal P \otimes k((t))^{perf}$ $(F_N)^0$ is $\frac{1}{\pi^{a_\tau}}F_N = (\id + [t]N) \frac{1}{\pi^{a_\tau}}F$ as $N$ is $\mathcal O$-linear, and this is $(F^0)_N$. The same is true for $V^{-1}$, thus we have the claim. Now if $D^0$ is not bi-infinitesimal, we have a decomposition
\[ D^0 = D^{0-et} \times D^{0-bi} \times D^{0,mult},\]
where $D^{0,mult} \simeq (D^{0,et})^\vee$ by the polarisation on $D^0$. But as $D^0 = (D)^0$, we have the asserted decomposition of $D$, and a direct calculation gives the slope of $D^{a-et}$. As passing from $\mathcal P$ to $\mathcal P_N$ does not change the Hodge filtration, we have the assertion on Hodge polygons.
%
\end{proof}

With the previous proposition we can explain our strategy. We start with a point in the Rapoport locus. As any deformation of it is still in the Rapoport locus, by the previous proposition for example, we will be able to lift the Pappas-Rapoport filtration canonically (cf \cref{thrRap}) and the deformation will still be in the Rapoport locus, thus we can forget about the Pappas-Rapoport datum for now. Then we will modify the crystal $M$ of our $p$-divisible group by the previous proposition, deform this crystal inductively by a display, and ultimately the non-necessarily parallel crystal deformation $N^0$ will not be infinitesimal anymore. Thus the associated deformation of $M$ will split, and by induction on the Newton polygon, we will be able to conclude. More precisely we can always write
\[ G = G^1 \times G^{00} \times (G^1)^D,\]
by Hodge-Newton decomposition (\cite{BH} Théorème 1.3.2), where $G^1$ is the biggest subgroup of $G$ such that the Hodge and Newton polygon of $G^1$ are equal. Remark that $G^{00}$ is still polarized (and in the Rapoport locus if $G$ is). The induction is on the height of $G^{00}$.

There will be a slight issue in the case where $a_\tau + a_{\overline\tau} = e$ for all $\tau$ in our method. Fortunately we have the following proposition
\begin{prop}
\label{propmuord}
Let $M$ be a polarised $\mathcal O$-crystal such that for all $\tau$, if $a_\tau$ denotes the first slope of $\Hdg_\tau$, $a_\tau + a_{\overline\tau} =e$. Then the $p$-divisible group associated to $M$ is $\mu$-ordinary.
\end{prop}

\begin{proof}
As $F_\tau V_\tau = pId_{M_\tau}$ for all $\tau$, this implies that $F^0_\tau$ and $V_\tau^0$ are invertible $\sigma$ and $\sigma^{-1}$ respectively $W(k)\otimes \mathcal O$-morphisms. Thus the Newton slopes of $F^0_{\sigma^{f-1}\tau}\circ\dots\circ F_\tau^0$ are all equal to one, and thus $F_{\sigma^{f-1}\tau}\circ\dots\circ F_\tau$ (which we write usually $F^f$) is isocline of slope $\frac{1}{e}\sum_\tau a_\tau$. Thus $G$ is isocline, and thus $\mu$-ordinary.
\end{proof}

Thus let $x_G \in X^{PR}$ be a point in the Rapoport locus, and denote $M$ the bi-infinitesimal part of its crystal. For now on by proposition \ref{propNNP}, we can suppose that $M$ is a non-necessarily parallel crystal, whose first slope for $\Hdg_\tau$ is zero for all $\tau$.
In particular this means that for all $\tau$, there exists
$x \in M_\tau$ such that 
\[ F^0(x) \not\equiv 0 \pmod{\pi}.\] 

By proposition \ref{propmuord}, we can moreover suppose that there exists $\tau_0$ such that $a_{\tau_0} + a_{\overline{\tau_0}} < e$ (otherwise $x_G$ is in the $\mu$-ordinary locus and we are done). Thus we have that $F_{\tau_0}^0V_{\tau_0}^0 \equiv 0 \pmod \pi$.

\begin{lemm}
\label{lemmB7}
Denote $f = 2d$ where $[F_\pi:\QQ_p] = ef$. Let $M/W(k)$ be a NNP polarised $\mathcal O$-crystal as before (bi-infinitesimal with first slope of $Hdg_\tau$ being zero).
Then there exists a deformation of $M$ to a (NNP polarised $\mathcal O$)-display $\mathcal P$ over $k[[X]]$ and $x \in M_{\tau_0}$ such that $F^{d}(x) \neq 0 \pmod\pi$.
\end{lemm}

\begin{proof}
Suppose that it is not the case, i.e. for all $x \in M_{\tau_0}$ $F^{d}(x) \equiv 0 \pmod \pi$. Take $x \in M_{\tau_0}$ such that $F(x) \not\equiv 0 \pmod \pi$.
Let $r$ be the maximal integer such that $F^r(x) \not\equiv 0 \pmod \pi$. $F^r(x) \in M_{\sigma^r\tau}$, and take $y \in M_{\sigma^r\tau}$ such that $F(y) \not\equiv 0 \pmod \pi$.
Thus $F^r(x)$ and $y$ are not colinear and are non-zero modulo $\pi$, thus we can construct an endomorphism $N_r$ of $M_{\sigma^r\tau}$ that is $\mathcal O_{\sigma^r\tau}$-linear and such that $N_r(F^r(x)) = y$ and $N_r(y) =0$, and $N^2_r = 0$.
Then set $N_{\overline{r}} = - N_r^* \in \End(M_{\sigma^r\overline{\tau}})$ and for every embedding $\chi \neq \sigma^r\tau, \sigma^r\overline\tau$, set $N_{\chi} = 0$.
$N$ is $\mathcal O$-linear, polarised and $N^2 = 0$. Now in $P_N = M\otimes_{W(k)} W(k[[X]])$, we can calculate,
\begin{eqnarray*} F_N^{r+1}(x \otimes 1) = F_N^2(F_N^{r-1}(x\otimes 1)) = F_N^2(F^{r-1}(x)\otimes 1) = F_N(F^r(x)\otimes 1 + y \otimes X) \\
= F^{r+1}(x) \otimes 1 + F(y) \otimes X \equiv XF(y) \not\equiv 0 \pmod \pi.\end{eqnarray*}
By induction, we can suppose that $F^d(x) \not\equiv 0 \pmod \pi$ up to deform $M$. 
\end{proof}

\begin{lemm}
Let $M$ as in the conclusion of the previous lemma. There exists a deformation of $M$ such that there exists $y \in M_{\overline{\tau_0}}$ satisfying
\[ F^d(y) \not\equiv 0 \pmod \pi.\]
Moreover there is still $x \in M_{\tau_0}$ such that $F^d(x) \not\equiv 0 \pmod \pi$. 
\end{lemm}

\begin{proof}
If it is not already the case for $M$, let $y \in M_{\overline{\tau_0}}$ such that $F(y) \not\equiv 0 \pmod \pi$ and denote $r$ the maximal integer such that 
$F^r(y) \not\equiv 0 \pmod \pi$. We will construct a deformation such that 
$F^{r+1}(y) \not\equiv 0 \pmod\pi$. Choose $z \in M_{\sigma^r\overline{\tau}}$ such that $F(z) \not\equiv 0 \pmod \pi$.
We then set $N$ as in the previous lemma : 
\[ N_{\sigma^r\overline\tau}(F^r(y)) = z, \quad N_{\sigma^r\overline\tau}(z) = 0, \quad N_{\sigma^r\tau} = - N_{\sigma^r\overline\tau}^*, \quad N_\chi = 0\quad \forall \chi \neq \sigma^r\tau,\sigma^r\overline\tau.\]
The same calculation shows that $F_N^{r+1}(y \otimes 1) \not \equiv 0 \pmod \pi$. Moreover, $F_N^d(x)$ reduces to $F^d(x)$ modulo $X$, thus $F_N^d(x)$ is still non zero modulo $\pi$ as $F^d(x)$ is.
\end{proof}

\begin{lemm}
\label{lemB9}
Let $M$ be as in the conclusion of Lemma \ref{lemmB7}. Then there exists a deformation such that $F^{2d}(x) \not\equiv 0 \pmod \pi$ for some $x \in M_{\tau_0}$.
\end{lemm}

\begin{proof}
If it is not already the case, let $x$ be the element given in \ref{lemmB7} and up to deform $M$ we can also have an element $y \in M_{\overline\tau}$ as in the previous lemma.
Then we can construct an $\mathcal O$-linear $N$ such that $N^2 = 0$ and 
\[ N_{\overline\tau}F^d(x) = y, N_{\overline\tau}(y) = 0, N_\tau = - N_{\overline\tau}^*, \quad N_\chi = 0 \quad \forall \chi \neq \tau,\overline\tau.\]
Set $(P,Q,F_N,V^{-1}_N)$ as in \cite{Wed1}. Then we can calculate,
\[ F_N^{2d}(x) = F^{2d}(x) + XNF^{2d}(x) + XF^d(y) + X^2NF^d(y).\]
But $F^{2d}(x) \equiv 0 \pmod \pi$, thus $NF^{2d}(x)\equiv 0 \pmod \pi$ as well by linearity of $N$. Moreover $F^d(y) \not\equiv 0 \pmod\pi$ thus 
\[F^d(y) + XNF^d(y) \not\equiv 0 \pmod \pi,\]
as it is the case modulo $X$, and thus $F_N^{2d}(x) \not\equiv 0 \pmod \pi$. 
\end{proof}

\begin{prop}
Let $x \in X^{PR}(k)$. Then there exists a sequence of generisations $x_i, i = 0,\dots,n$ such that $x_i \in X^{PR}(k_i[[X]])$ for all $i = 1,\dots,n$ and some perfect field $k_i$ above $k_{i-1}((X))$ with $k_1 = k$ and $x_0 = x$, $x_i \pmod X = x_{i-1} \otimes_{k_{i-1}[[X]]} k_i$, and $x_n \otimes_{k_n[[X]]} k_n((X))$ is $\mu$-ordinary.
 \end{prop}
 
 \begin{proof}
 By theorem \ref{thrRap}, we can suppose that we have constructed $x_1$ and $x_1 \otimes k((X))$ is in the Pappas-Rapoport locus. We will proceed by induction on the number of slopes of the bi-infinitesimal part of $x_i$ already constructed. If $x_i$ has no or only one slope for its Newton polygon, we are done as it is $\mu$-ordinary. Otherwise we can always assume that $G_{x_i}$ is split,
 \[ G_{x_i} = G_{x_i}^1 \times G_{x_i}^{00} \times G_{x_i}^2,\]
 with Hodge and Newton polygons of $G_{x_1}^1$ being equal, $G_{x_i}^2 = (G_{x_i}^1)^D$ and the first slopes of the Newton and Hodge polygons of $G_{x_i}^{00}$  differ.
 By the previous results, we can moreover suppose that there exists $r > 0$ and we have constructed $x_1,\dots,x_r$ such that $M^0 := M(G_{x_r}^{00})^0$ satisfies the conclusion of lemma \ref{lemB9}, and denote $x \in M^0_{\tau_0}$ such that $(F^0)^{2d}(x) \not\equiv 0 \pmod \pi$. 
 We can always suppose that
 \begin{equation}\label{eqpol}
 h (x, (F^0)^d(x)) \equiv 0 \pmod \pi).\end{equation}
 Indeed, if it is not the case, then for all $m \in M_{\sigma\tau_0}$,
 \[ h(x + V^0m,(F^0)^d(x)) = h(x,(F^0)^d(x)) + h(V^0m,(F^0)^d(x)) = h(x,(F^0)^d(x)) +h(m,(F^0)^{d+1}(x))^{\sigma^{-1}}.\]
 As $(F^0)^{d+1}(x) \neq 0 \pmod \pi$, there exists $m$ such that the previous expression vanishes. Replacing $x$ by $x+V^0m$, which doesn't change $F^0(x)$ as $F_{\tau_0}^0V_{\tau_0}^0 \equiv 0 \pmod \pi$\footnote{here we use that $a_{\tau_0} + a_{\overline{\tau_0}} < e$}, we are done.
 Denote $s$ the maximal integer such that $(F^0)^s(x) \not\equiv 0 \pmod \pi$, and write 
 \[ s = 2dq + j, \quad 0 \leq j < 2d.\]
 Set $m_0 = F^{2(q-1)d + j}(x)$. Then the sequence $(m_0,F^0(m_0),\dots, (F^0)^{2d-1}(m_0))$ is a deformation sequence. 
 Moreover, 
 we have by applying $(F^0)^{2(q-1)d+j+i}$ to (\ref{eqpol}), and by semilinearity of $h$ on the left,
 \[  h((F^0)^i(m_0),(F^0)^{i+d}(m_0)) = 0 \quad \forall i  \in\{0,\dots,d\}.\]
We can now follow \cite{Wed1} to construct $N$ and thus a deformation of $M^0$ which is not infinitesimal anymore. As $h(F^{2d}(m_0),F^d(m_0)) = 0$
 the subspace $\Vect(F^{2d}(m_0),m_0,F^d(m_0)) \pmod \pi \subset M_{\sigma^j\tau}/\pi \oplus M_{\sigma^j\overline\tau}/\pi$ is totally isotropic. We can thus find $U$ a totally isotropic complement, $U = U_{\sigma^j\tau}\oplus U_{\sigma^j\overline\tau}$, such that $h( . ,. )$ induces a perfect pairing between 
\[ M_0 := \Vect(F^{2d}(m_0),m_0,F^d(m_0)) \pmod \pi \quad \text{and} \quad U.\]
We can then set $N(F^{2d}(m_0)) = m_0, Nm_0 = NF^d(m_0) = 0$, and extends uniquely $N$ to $U$ such that $N$ is skew-symmetric. Then $N^2 = 0 \pmod \pi$ and we can extend $N$ by zero on $(M_0 \oplus U)^\perp$ and lift to $M$ so that $N^2 = 0$, and is still skew symmetric and $\mathcal O$-linear.
Then we can calculate that for the deformation $P_N$, 
\[ (F_N^0)^{2d}(m_0\otimes 1) = F_N^0 ((F^0)^{2d-1}(m_0) \otimes 1) = (F^0)^{2d}(m_0)\otimes 1 + m_0\otimes X \equiv Xm_0\pmod{\Ker(F \pmod\pi)}.\]
In particular $F_N^0$ is not nilpotent, and thus $N^0 = \mathcal P_N^0 \otimes k((X))^{perf}$ is not bi-infinitesimal. The $p$-divisible group associated to $P_N$ gives a $k[[X]]$-point $x_{r+1}$ of $X^{PR}$ such that $G_{x_{r+1}} \otimes_{k[[X]]} k = G_{x_r}\otimes_{k_r[[X]]} k$ and $G_{x_{r+1}} \otimes k((X))$ is split,
\[ G_{x_{r+1}} = G_{x_{r+1}}^1 \times G_{x_{r+1}}^{00} \times G_{x_{r+1}}^{1,D},\]
with $G_{x_{r+1}}^{00}$ having height less than $G_{x_i}^{00}$. By induction on this height, we get the result.
 \end{proof}

\begin{cor}
In case (AU), the $\mu$-ordinary locus $X^{PR}_{\mu-ord}$ is Zariski dense.
\end{cor}

\subsection{Case C}

In this case, $M = \oplus_{\tau} M_\tau$ and \[ h : M \times M \fleche W(k) \otimes \Diff^{-1}_F,\] is $\mathcal O_F$-linear and alternating. 
Let $x_G \in X(k)$ a point corresponding to a group $G$ (with $k$ perfect). We suppose that $x_G$ is in the Rapoport locus (by \ref{thrRap}, as in this case Rapoport and generalised Rapoport loci coincide).

\begin{lemm}
If $G$ is bi-infinitesimal and in the generalized Rapoport locus, there exists a deformation sequence (as in definition \ref{defseq}).
\end{lemm}

\begin{proof}
As $G$ is bi-infinitesimal, denote for $x \in M$, $w(x) = \sup \{ n | F^n(x) \not\equiv 0 \pmod{\pi}\}$. As $G$ is in the generalized Rapoport locus (and in case $C$ we have $d_\tau = \frac{h}{2}$, 
for all $\tau$), 
we have for all $\tau$, a $x_\tau \in M_\tau$ such that $w(x_\tau) \geq 1$.
Then, this is proved exactly as in \cite{Wed1}, Proposition 4.1.4.
\end{proof}

\begin{lemm}
There exists a deformation $G'$ of $G$ such that $G'$ is not bi-infinitesimal.
\end{lemm}

\begin{proof}
If $G$ is not bi-infinitesimal, any deformation will do. Otherwise, let $(x_\tau)_\tau$ be the deformation sequence given by the previous lemma.
We can construct a deformation endomorphism $N$ such that $NFx_\tau = 0$ if $Fx_\tau = x_{\sigma\tau}$ and $NFx_\tau = x_{\sigma\tau}$ otherwise.
Indeed, as $h$ is $\mathcal O_F$-linear, this is done exactly as in \cite{Wed1} Proposition 4.4.3. The calculation of $F_N$ using this deformation shows that $G'$, the defomation of $G$ associated to $N$, is not bi-infinitesimal.
%
%
%
%
\end{proof}

\begin{prop}
In case (C), the ordinary locus is dense.
\end{prop}

\begin{proof}
By \cref{thrRap}, it suffices to prove that each $x \in X^{PR}(k)$, $k$ a perfect field, can be deformed into a ordinary $p$-divisible group.
We can split
\[ G = G^{m} \times G^{00} \times G^{et},\]
where $G^{00}$ is bi-infinitesimal. We will argue on the height of $G^{00}$. If $G^{00}$ is trivial, then $G$ is ordinary and we are done.
Otherwise by the previous lemma, there is a deformation $H$ of $G^{00}$ which is not bi-infinitesimal, thus $\widetilde G = G^m \times H \times G^{et}$ is a deformation of $G$ and
\[ \Ht_{\mathcal O} \widetilde G^{00} = \Ht_{\mathcal O} H^{00} < \Ht_{\mathcal O} G^{00}.\]
Thus by induction, there exists a deformation of $G$ which is ordinary.
\end{proof}

\subsection{Ordinary Locus}

\begin{defin}
A $p$-divisible group over a base of characteristic $p$ is said to be ordinary if it is an extension of an etale group by a multiplicative one. Equivalently, it is ordinary if its Hasse invariant is invertible. Denote by $X^{ord}$ the (open) subset of $X_\kappa$ of ordinary $p$-divisible group. 
\end{defin}

\begin{prop}
We have the following properties;
\begin{enumerate}
\item If the ordinary locus is non-empty, it is equal to the $\mu$-ordinary locus and is thus dense.
\item The ordinary locus is non-empty if and only if $(d_{\tau}^{[i]})$ is constant for all $\tau,i$.
\item The ordinary locus is non-empty if and only if the local reflex field $E$ is equal to $\QQ_p$.
\end{enumerate}
\end{prop}

\begin{proof}
If $(d_{\tau}^{[i]})$ is constant, say equal to $d$, then the Pappas-Rapoport Polygon has slopes 0 ($d$-times) and 1 ($(h-d)$ times). In particular,
\[ X^{\mu-ord} = X^{ord}.\]
If the ordinary locus is non-empty, then a point $x$ corresponding to an ordinary $p$-divisible group and has a Newton polygon with slopes only $0$ and $1$, and same ending point as $\PR$, thus as $\Newt(x) \geq \PR$, this means that $\PR(d_{\tau}^{[i]})$ has only slopes $0$ and $1$, and thus (as the breaking points are at the abscissa $d_{\tau}^{[i]}$) the collection $(d_{\tau}^{[i]})$ is constant. This proves 1. and 2.
$E$ is the (finite) extension of $\QQ_p$, inside $K$, fixing the collection $(d_\tau^{[i]})$. Thus, if the ordinary locus is non empty, $E = \QQ_p$. 
For every $\sigma \in \Gal(K/K^0)$, $\sigma \cdot d_\tau^{[i]} = d_\tau^{[\sigma\cdot i]}$ where $i$ correspond to a conjugate $\pi_i$, of $\pi$, and $\pi_{\sigma \cdot i} = \sigma(\pi_i)$. Thus $\Gal(K/K^0)$ is transitive on the collection $(d_\tau^{[i]})_i$. Thus if $E = \QQ_p$, $d_\tau^{[i]} = d_\tau$ for all $i$. But $\Gal(K^0/\QQ_p)$ is transitive on the set $\mathcal T$, and thus if $E = \QQ_p$, $d_\tau^{[i]} = d$ for all $\tau,i$.
Another way to say it is that, using the characteristic zero description $d_{i,\tau'}$, $\tau' \in Hom(K,\overline \QQ_p)$, $\Gal(K/\QQ_p)$ acts transitively on $Hom(K,\overline \QQ_p)$.
\end{proof}


\appendix

\section{A specific exemple in case (AR)}
\label{AppE}
In this section we give explicite calculations for the local rings of the Pappas-Rapoport model for $U(1,1)$ and $U(2,1)$ and a quadratic extension which $p \neq 2$ is ramified. This setting has been studied (in slightly greater generality) in \cite{Kramer}. Thus we fix $p \neq 2$, $F/\QQ_p$ a ramified extension of degree 2, with uniformiser $\pi$ and denote $\overline \pi = s(\pi)$ its conjugate.

\subsection{The case of U(1,1)}

For $U(1,1)$ the moduli problem $\mathcal{PRZ}$ (local analog of the definition of $X$ \ref{defPRmoduli}) with values in a $\mathcal O_F$-scheme $S$ is given by 
\begin{itemize}
\item $G$ a $p$-divisible $\mathcal O_F$-module over $S$ of $\mathcal O_F$-height 2, dimension 2, and denote $\iota : \mathcal O_F \fleche \End(G)$,
\item a polarisation, i.e. an isomorphism $G^D \simeq G^{(s)}$,
\item a locally direct factor $\omega^{[1]} \subset \omega_G$ of rank 1, such that \[(\iota(\pi)\otimes 1 - 1 \otimes \pi)\omega^{[1]} = \{0\};\]
\[ (\iota(\pi)\otimes 1 - 1 \otimes s(\pi))(\omega_G) \subset \omega^{[1]}.\]
\end{itemize}
In characteristic $p$, we have 
\[ \omega^{[1]} \subset \omega_G \subset \mathcal H^1_{dR}(G) = \mathcal H^1_{dR}(G)[\pi^2].\]
We can thus look at $\pi^{-1}\omega^{[1]}$, which contains $\omega_G$ by hypothesis on $\omega^{[1]}$, and is locally free of rank $3$. Thus
\[ \omega^{[1]'} := (\pi^{-1}\omega^{[1]})^\bot,\]
where $\bot$ denotes the orthogonal with respect to the perfect pairing on $\mathcal H^1_{dR}$ induced by the polarisation, is locally free of rank 1, and inside $\omega_G$.

The associated local model $\mathcal M$ is given by $(\mathcal F^{[1]} \subset \mathcal F)$ in $\Lambda \otimes_{\mathcal O_E} \mathcal O_S = \mathcal O_F^2 \otimes_{\ZZ_p} \mathcal O_S$, endowed with (say) the pairing in the basis 
$\pi e_1,e_1,\pi e_2,e_2$,
\[
J = \left(
\begin{array}{cccc}
  0   & 1 & & \\
   -1 &  0 & & \\
         &  & 0 & 1\\
    &   &-1 & 0\\
\end{array}
\right)
\]
satisfying analogous conditions (see \cite{Kramer} Definition 4.1).
The induced pairing on $\mathcal H^1_{dR}[\pi]$ or $\Lambda/\pi\Lambda$, is given by $\tilde J(\pi e_i,\pi e_j) = J(\pi e_i,e_j)$ and thus by the matrix $I_2$.
To understand locally the moduli space $\mathcal{PRZ} \otimes \overline{\FP}$, we can make the calculation on the local model $\mathcal M$. As $\Lambda/p\Lambda$ is of rank 
4 over $\ZZ_p$, this amounts to understand the possible inclusions $\mathcal F^{[1]} \subset \mathcal F \subset \Lambda \otimes \mathcal O_S$ and their deformations.
We will fix once and for all the basis $\pi e_1,\pi e_2,e_1,e_2$ of $\Lambda$ and identify the points of $\mathcal M \otimes \overline{\FP}$ with 4x2 matrices, first column generating $\mathcal F^{[1]}$ and first two generating $\mathcal F$.

Up to obvious symmetries, a point of $\mathcal M$ is given by
\[
\omega = \left(
\begin{array}{cc}
  1   & 0  \\
   x &  a \\
         &b  \\
    & y  \\
\end{array}
\right),
\]
and as $\pi \mathcal F \subset \mathcal F^{[1]}$, we must have $bx =y$, and $\mathcal F$ is totally isotropic thus $b + xy = 0$, i.e $b(1+x^2) = 0$.
Thus either $b = 0$ and we have 
\[\omega = \left(
\begin{array}{cc}
  1   & 0  \\
   x &  1 \\
         &0  \\
    & 0  \\
\end{array}
\right),
\] which is not in the generalised Rapoport locus (here this is just the Rapoport locus) as $\omega$ is $\pi$-torsion. Or $b \neq 0$ and thus
\[\omega_{PR} = \left(
\begin{array}{cc}
  1   & 0  \\
   x &  a \\
         &1  \\
    & x  \\
\end{array}
\right),
\] 
Thus, $\mathcal M \otimes \overline{\FP}$ is locally given by two lines $L_{b=0}$ and $L_{1+x^2=0} = 0$ intersecting at a point outside of the Rapoport locus,
\[x_0 = \left(
\begin{array}{cc}
  1   & 0  \\
   x &  1 \\
         &0  \\
    & 0  \\
\end{array}
\right),
\] such that $1+x^2 = 0$. Remark that $1+x^2 = 0$ is exactly the condition so that $\mathcal F^{[1]'} = \mathcal F^{[1]}$, i.e. $\mathcal F^{[1]}$ is totally isotropic for the induced pairing on $\Lambda/\pi$. $L_{1+x^2}$ is the closure of the Rapoport locus, and $L_{b=0}$ is completely away from the Rapoport locus. In particular, the (generalised) Rapoport locus is not dense (and thus so is the ($\mu$-)ordinary locus). The local ring at $x_0$ is given by
\[ (\overline{\FP}[A,B,X]/(B(1+X^2)))_{(B,X-x,A-a)}.\]

\subsection{The case of U(2,1)}

The problem is similar, in this case we define $\mathcal M$ parametrizing $\mathcal F^{[1]} \subset \mathcal F \subset \Lambda \otimes \mathcal O_S$ locally direct factor of ranks 2 and 3, $\mathcal F$ being totally isotropic, satisfying analogous assumptions with respect to $\pi$, and $\Lambda = \mathcal O_F^3$ with the pairing given in the basis $(\pi e_1,e_1,\pi e_2,e_2,\pi e_3,e_3)$,
\[
J = \left(
\begin{array}{cccccc}
  0   & 1 & && & \\
   -1 &  0 & & \\
         &  & 0 & 1& &\\
    &   &-1 & 0& &\\
  & &           &  & 0 & 1\\
   & & &   &-1 & 0\\
\end{array}
\right).
\]
Looking at points of $\mathcal M \otimes \overline{\FP}$ as matrices in the basis $(\pi e_1,\pi e_2,\pi e_3,e_1,e_2,e_3)$ we see that, up to obvious symmetries, 
\[
\omega = \left(
\begin{array}{ccc}
  1   &  & 0  \\
   0 & 1  & 0 \\
    x     & y & a  \\
  & &    b  \\
    & &    c  \\
      & &    d  \\
\end{array}
\right),
\]
with $bx + cy = d$ (as $\pi \mathcal F \subset \mathcal F^{[1]}$) and $b + xd = 0$ and $c + yd = 0$ (as $\mathcal F$ is totally isotropic). This amounts to variables $x,y,a,d$ and equation $d(1 + x^2 + y^2)$. Thus as before we have two smooth surfaces (given by $d = 0$ – when $\omega$ is $\pi$-torsion, and $1+x^2+y^2$ when $\mathcal F^{[1]'}$ is totally isotropic for the induced pairing), intersection along a smooth curve (given by $d = 1+x^2+y^2 = 0$). Moreover, for any point $z$ on the curve, the local ring at $z$ is given by
\[ (\overline{\FP}[X,Y,A,D]/(D(1+X^2+Y^2)))_{(D,X-x,Y-y,A-a)}.\]
In this case the surface $S : 1+x^2+y^2 = 0$ contains the generalised Rapoport locus as a dense subset (corresponding to $d \neq 0$) and coincides with its closure, and the other surface is completely disjoint from the generalised Rapoport locus. In particular, Theorem \ref{thr229}, \ref{thrRap} (thus also theorem \ref{thr41} and proposition \ref{prop38}) are false in this example too, as in the previous one.

\bibliographystyle{smfalpha} 
\bibliography{biblio} 

\end{document}